\documentclass[11pt]{amsart}
\usepackage[margin=1.1in]{geometry}
\usepackage{amsmath,amssymb,amsthm}
\usepackage[colorlinks=true,linkcolor=blue,citecolor=blue,urlcolor=blue]{hyperref}

\newtheorem{theorem}{Theorem}[section]
\newtheorem{proposition}[theorem]{Proposition}
\newtheorem{lemma}[theorem]{Lemma}
\newtheorem{corollary}[theorem]{Corollary}
\theoremstyle{remark}
\newtheorem{remark}[theorem]{Remark}
\theoremstyle{definition}

\newcommand{\PP}{\mathbb{P}}
\newcommand{\ZZ}{\mathbb{Z}}
\newcommand{\QQ}{\mathbb{Q}}
\newcommand{\NN}{\mathbb{N}}
\newcommand{\CC}{\mathbf{C}}

\newcommand{\bk}{\mathbf{k}}
\newcommand{\cO}{\mathcal{O}}

\newcommand{\cE}{\mathcal{E}}
\newcommand{\cF}{\mathcal{F}}
\newcommand{\cG}{\mathcal{G}}
\newcommand{\cQ}{\mathcal{Q}}

\newcommand{\cK}{\mathcal{K}}
\newcommand{\NS}{\operatorname{NS}}
\newcommand{\rk}{\operatorname{rk}}
\DeclareMathOperator{\depth}{depth}
\DeclareMathOperator{\Proj}{Proj}
\DeclareMathOperator{\Spec}{Spec}
\DeclareMathOperator{\Supp}{Supp}
\DeclareMathOperator{\Tot}{Tot}
\DeclareMathOperator{\Av}{Av}
\DeclareMathOperator{\Sym}{Sym}

\title[Varieties without aCM bundles, rings without graded MCM modules]{Polarized varieties without arithmetically Cohen--Macaulay bundles and section rings without graded maximal Cohen--Macaulay modules in characteristic zero}
\author{Cristian Anghel}
\address{Institute of Mathematics of the Romanian Academy, 21 Calea Grivi\c{t}ei, 010702 Bucharest, Romania}
\email{cristian.anghel@imar.ro}
\date{September 2026}

\begin{document}

\begin{abstract}
We exhibit smooth polarized surfaces $(Y,H)$ over $\CC$ carrying no nonzero arithmetically Cohen--Macaulay bundle of any rank. Equivalently, their section rings are three-dimensional normal $\NN$-graded $\CC$-domains, with an isolated singularity at the vertex, admitting no nonzero finitely generated graded maximal Cohen--Macaulay module. To the author's knowledge no such ring was previously known. Their existence contrasts with the theorem of Hartshorne, Hochster and Peskine--Szpiro that a three-dimensional $\NN$-graded domain over a perfect field of characteristic $p>0$ always has one. As a consequence, in characteristic zero Hochster's small Cohen--Macaulay conjecture admits no graded refinement.

The nonexistence criterion is numerical: for $|H|$ base-point-free with finite morphism, $H^2<K_Y^2-8\chi(\cO_Y)=\tau(Y)$, the signature of $Y$ over $\CC$, leaves no nonzero $\cE$ with $H^1(Y,\cE(tH))=0$ for all $t$. For Ulrich bundles, being semistable, this is Bogomolov's inequality; what is new is that no stability hypothesis and no condition on the Hilbert polynomial are needed, the Harder--Narasimhan filtration of $\cE$ being compared instead with the Horrocks splitting of its direct image on $\PP^2$. Hirzebruch's Hesse surfaces with $H_n=4A_n-E_n$, $n\ge3$, qualify. In dimension $m$ the obstruction reads $(m+1)H^m\ge(K_X^2-2c_2(X))\cdot H^{m-2}$, the inequality Lopez proved for Ulrich bundles; products and complete intersections inside them give examples in every dimension $\ge3$, and the threefold examples show that the characteristic-$p$ hypothesis in the criterion of Shimomoto--Tavanfar is essential.
\end{abstract}

\maketitle

\section*{Introduction}

Hochster's small Cohen--Macaulay conjecture asks whether every complete local ring admits a nonzero finitely generated module whose depth equals the dimension of the ring, a \emph{maximal Cohen--Macaulay} (MCM) module \cite{Ho75,Ho04}. It holds in dimension $\le2$ and is open in dimension $\ge3$, in characteristic zero even for localizations of affine domains at maximal ideals \cite{ShTa}. In dimension three and characteristic $p>0$ it is known for complete weakly $F$-split rings \cite{Sch17}, for complete local domains with algebraically closed residue field admitting a nonzero $F$-invariant integrable derivation \cite[Thm.~3.4]{Sch20}, for pseudo-graded rings \cite{Sch20}, and, in graded form, for graded rings: a three-dimensional $\NN$-graded domain finitely generated over a perfect field of characteristic $p$ has a \emph{graded} maximal Cohen--Macaulay module. This theorem was obtained independently by Hartshorne, Hochster and Peskine--Szpiro, as recorded in \cite[Thm.~1 and Cor.~2]{Ho04}; see \cite{Ho75,PS} and the account in \cite{ShTa}.

The purpose of this note is to show that the graded statement has no analogue in characteristic zero: there are three-dimensional normal $\NN$-graded $\CC$-domains, with an isolated singularity at the vertex, which admit no nonzero finitely generated graded maximal Cohen--Macaulay module (Corollary~\ref{cor:C}). To the author's knowledge no such example was known, in either of the two equivalent formulations. In the related examples of \cite[Thm.~3.4 and Rem.~3.6]{Ma19}, rank-one maximal Cohen--Macaulay modules are excluded but modules of higher rank exist; here every rank is excluded. Geometrically the statement reads: there are smooth polarized surfaces carrying no nonzero arithmetically Cohen--Macaulay bundle of any rank (Theorem~\ref{thm:A}). The two formulations are equivalent here (Lemma~\ref{lem:mcm} below), but they are not equally expected: on the surfaces for which aCM bundles have been studied they are abundant, and $\cO_X$ is one on every arithmetically Cohen--Macaulay embedded variety, so the vanishing of the whole category is the surprising half of the statement (Remark~\ref{rem:ulrich}).

The examples are section rings $R(Y,H)=\bigoplus_{t\ge0}H^0(Y,\cO_Y(tH))$ of polarized smooth surfaces, and the mechanism is geometric. When $|H|$ is base-point-free with finite associated morphism, graded maximal Cohen--Macaulay modules over $R(Y,H)$ correspond exactly to the vector bundles $\cE$ on $Y$ with $H^1(Y,\cE(tH))=0$ for all $t\in\ZZ$, the \emph{arithmetically Cohen--Macaulay} (aCM) bundles with respect to $H$ (Lemma~\ref{lem:mcm} below; for $H$ very ample this is classical, cf.\ \cite{GoWa,EGA3}; the correspondence is characteristic-free; the new input for the nonexistence conclusion is Theorem~\ref{thm:A}). For ample $\QQ$-divisors in characteristic $p$, Shimomoto--Tavanfar give a different existence criterion: the existence of a graded maximal Cohen--Macaulay module is characterized by the existence of a coherent sheaf with maximal Cohen--Macaulay stalks \cite[Cor.~4.5(i)]{ShTa}. This is a local criterion, not the aCM-bundle correspondence used here, and it is satisfied by $\cO_X$ on every smooth $X$; Corollary~\ref{cor:shta} below shows that its nontrivial implication fails in characteristic zero. That corollary is stated for smooth threefolds, rather than for the surfaces of Corollary~\ref{cor:C}, only because \cite[Cor.~4.5]{ShTa} is stated for $\dim X\ge3$.
The main result of this note, Theorem~\ref{thm:A}, assumes no stability property whatever: for $H$ ample and base-point-free, if $H^2<K_Y^2-8\chi(\cO_Y)$, then $(Y,H)$ carries no aCM bundle at all, of any rank. For a \emph{semistable} bundle the same conclusion is a direct consequence of Bogomolov's inequality, and is the argument by which \cite{An} excludes Ulrich bundles, these being automatically semistable; what Theorem~\ref{thm:A} adds is the unstable case (Remark~\ref{rem:rr}). Its proof pushes $\cE$ forward to $\PP^2$ by a finite projection, where Horrocks' theorem splits it, and compares the Harder--Narasimhan filtration of $\cE$ with the splitting type through a majorization inequality; Bogomolov's inequality enters only through the Harder--Narasimhan factors, and is the single ingredient of the proof that fails in characteristic $p$. In the language of Boij--S\"oderberg theory, to which the Ulrich question of \cite{ES} was connected later, the cone of cohomology tables of vector bundles on such a pair contains no nonzero table with vanishing $h^1$-row (Remark~\ref{rem:bs}).

Over $\CC$ the number $K_Y^2-8\chi(\cO_Y)$ is the signature $\tau(Y)$ of the intersection form on $H^2(Y,\ZZ)$, so the obstruction reads: \emph{if $(Y,H)$ carries an aCM bundle, then the degree of $(Y,H)$ is at least the signature of $Y$}. By the classification of surfaces, $\tau(Y)>0$ only for $Y=\PP^2$ and for surfaces of general type; since an ample divisor on $\PP^2$ has $H^2\ge1=\tau(\PP^2)$, the obstruction can only occur on surfaces of general type of positive signature, and it is strongest, for a given $\chi(\cO_Y)$, on ball quotients (Remark~\ref{rem:signature}). The Hesse surfaces $Y_n$ of Hirzebruch \cite{Hi,BHH}, the $(\ZZ/n)^{11}$-covers of the plane blown up at the nine flexes of the Hesse pencil, branched along the twelve lines of the Hesse arrangement and the nine exceptional curves, carry the divisor $H_n=4A_n-E_n$ of \cite{An}, with $H_n^2=7n^9$, while $\tau(Y_n)=(3n^2-11)n^9$. The base-point-freeness of $H_n$ and the finiteness of its morphism, which is all that Theorems~\ref{thm:A} and \ref{thm:B} require, are proved in Section~\ref{sec:example} from the structure of the cover, so that the statements of Corollary~\ref{cor:C} about the section rings are independent of \cite{An} (Remark~\ref{rem:imported}); the very ampleness of $H_n$ proved in \cite{An} only identifies the standard graded ring $A(Y_n,H_n)$ of Theorem~\ref{thm:B} with the homogeneous coordinate ring of $Y_n$ itself. In scope the two papers differ: \cite{An} constructs the pairs $(Y_n,H_n)$ and excludes Ulrich bundles on them and on further families of covers, whereas Theorem~\ref{thm:A} applies to aCM bundles of arbitrary rank, hence to all graded maximal Cohen--Macaulay modules over the section rings, and, through Proposition~\ref{prop:higher}, in every dimension. The second Hirzebruch family also gives quasi-Gorenstein section rings with the same nonexistence property (Corollary~\ref{cor:qgor}).

The same numerical obstruction was proved for Ulrich bundles, in any dimension, by Lopez \cite[Lem.~3.1]{Lo}; restricting to a smooth surface and applying Theorem~\ref{thm:A} extends it to aCM bundles: a bundle $\cE$ on a smooth projective $m$-fold $X$ with $H^i(X,\cE(tH))=0$ for $1\le i\le m-1$ and all $t$ forces $(m+1)H^m\ge(K_X^2-2c_2(X))\cdot H^{m-2}$ (Proposition~\ref{prop:higher}), and the products $Y_n\times\PP^{m-2}$ give normal graded domains of every dimension $\ge3$ without graded maximal Cohen--Macaulay modules (Corollary~\ref{cor:products}), the second factor being replaceable by any smooth variety with a base-point-free polarization (Remark~\ref{rem:factors}); general complete intersections of members of $|a_1H|,\dots,|a_rH|$ in these products give further examples in every dimension, and surfaces other than the $Y_n$ whose section rings have the same property (Corollary~\ref{cor:ci} and Remark~\ref{rem:ci}). The rings $R(Y_n,H_n)$ are generalized Cohen--Macaulay, as shown in the proof of Corollary~\ref{cor:products} with $m=2$. We stress that the conjecture in its local form is not refuted: Hochster's conjecture concerns the localization at the vertex, or its completion, where modules need not be graded. Whether an ungraded maximal Cohen--Macaulay module over the local ring would force a graded one is not addressed here.

Section~\ref{sec:statements} states the results and collects what is needed about section rings; Sections~\ref{sec:proofA} and \ref{sec:proofB} prove Theorems~\ref{thm:A} and \ref{thm:B}; Section~\ref{sec:example} treats the Hesse surfaces; Section~\ref{sec:higher} extends the results to higher dimension; Section~\ref{sec:remarks} discusses Ulrich bundles and modules and the consequences for cones of cohomology tables.

\section{Statements}\label{sec:statements}

Throughout, $\bk$ is an algebraically closed field of characteristic zero, $Y$ is a smooth connected projective surface over $\bk$, and $H$ is a divisor on $Y$ such that the linear system $|H|$ is base-point-free and the associated morphism $\varphi_H\colon Y\to\PP^{h_0-1}=\PP\bigl(H^0(Y,\cO_Y(H))^\vee\bigr)$, $h_0:=h^0(Y,\cO_Y(H))$, is finite (for instance $H$ very ample); then $H=\varphi_H^*\cO(1)$ is ample. Equivalently, $H$ is ample and base-point-free: for $H$ ample a curve contracted by $\varphi_H$ would have $H\cdot C=0$, so $\varphi_H$ has finite fibres and, being projective, is finite. Put
\[
d:=H^2,\qquad k:=K_Y\cdot H,\qquad \chi:=\chi(\cO_Y),\qquad \sigma:=K_Y^2-8\chi ,
\]
and $\cE(t):=\cE\otimes\cO_Y(tH)$. The degree $d$ is that of the polarized pair, $d=\deg\varphi_H\cdot\deg\varphi_H(Y)$, not that of the image. Over $\bk=\CC$ the number $\sigma$ is the signature of $Y$: by Noether's formula $c_2(Y)=12\chi-K_Y^2$ and Hirzebruch's signature theorem,
\[
\tau(Y)=\frac{K_Y^2-2c_2(Y)}{3}=K_Y^2-8\chi(\cO_Y)=\sigma ,
\]
so $\sigma=b^+(Y)-b^-(Y)$ is a topological invariant of $Y$, and the hypothesis $d<\sigma$ reads: \emph{the degree of $(Y,H)$ is smaller than the signature of $Y$}. Over an arbitrary $\bk$ we still write $\tau(Y):=K_Y^2-8\chi(\cO_Y)$. A vector bundle $\cE$ on $Y$ is \emph{arithmetically Cohen--Macaulay} (aCM) with respect to $H$ if $H^1(Y,\cE(t))=0$ for every $t\in\ZZ$. This is a global condition on $(Y,H)$; it must not be confused with the local Cohen--Macaulay property, which every vector bundle on a smooth variety has. (A torsion-free sheaf with the same vanishing is automatically locally free, Remark~\ref{rem:tf}.)

\begin{theorem}\label{thm:A}
If $(Y,H)$ carries a nonzero aCM vector bundle, then $H^2\ge K_Y^2-8\chi(\cO_Y)$; over $\CC$: the degree of $(Y,H)$ is at least the signature of $Y$. Equivalently, if $H^2<K_Y^2-8\chi(\cO_Y)$, there is no nonzero aCM vector bundle on $(Y,H)$.
\end{theorem}

\subsection*{The section ring}
Let $R=R(Y,H):=\bigoplus_{t\ge0}H^0(Y,\cO_Y(tH))$ be the section ring, $\mathfrak m=R_+$ its irrelevant ideal, and $\Gamma_*(\cF):=\bigoplus_{t\in\ZZ}H^0(Y,\cF(t))$ for a coherent sheaf $\cF$. For $s\in R_1$ let $Y_s=\{s\ne0\}\subset Y$. When $H$ is very ample we also consider the homogeneous coordinate ring $A=A(Y,H)$ of $Y\subset\PP^{h_0-1}$, the image of $\Sym H^0(Y,\cO_Y(H))\to R$; it is a graded subring of $R$ generated in degree $1$, with $A_1=R_1$ and $\Proj A=Y$. The ring $R$ need not be generated in degree $1$; the following lemma collects what replaces this.

\begin{lemma}\label{lem:ring}
\textup{(i)} There is $m_0$ with $R_m=R_1R_{m-1}$ for all $m\ge m_0$. Consequently $R$ is a finite module over its subring $\bk[R_1]$ generated by $R_1$, which is the homogeneous coordinate ring of the surface $\varphi_H(Y)\subset\PP^{h_0-1}$; $R$ is a finitely generated graded $\bk$-domain of dimension $3$; and $\sqrt{R_1R}=\mathfrak m$.

\textup{(ii)} For $s\in R_1$ the open set $Y_s$ is affine, with $\Gamma(Y_s,\cF)=(\Gamma_*(\cF)_s)_0$ for every quasi-coherent $\cF$; the $Y_s$ cover $Y$, and the isomorphisms $Y_s\cong\Spec(R_s)_0=D_+(s)$ glue to $Y\cong\Proj R$. Moreover $R_s=(R_s)_0[s,s^{-1}]$ and $(M_s)_t=s^t(M_s)_0$ for every graded $R$-module $M$, so that $\cO_{\Proj R}(t)=\cO_{\Proj R}(1)^{\otimes t}=\cO_Y(tH)$ for all $t\in\ZZ$ and $\widetilde{M(t)}=\widetilde M\otimes\cO_Y(tH)$.

\textup{(iii)} The punctured spectrum $U:=\Spec R\smallsetminus\{\mathfrak m\}$ is the scheme $\mathbf{Spec}_Y\bigl(\bigoplus_{t\in\ZZ}\cO_Y(tH)\bigr)$, the complement of the zero section in $\Tot(\cO_Y(-H))$. In particular $U$ is smooth, and $R$ is normal.
\end{lemma}

\begin{proof}
(i) Let $s_1,\dots,s_{h_0}$ be a basis of $R_1$ and $\cK$ the kernel of the surjection $\cO_Y^{\oplus h_0}\to\cO_Y(H)$ they define ($|H|$ is base-point-free); $\cK$ is locally free, and Serre vanishing gives $H^1(Y,\cK((m-1)H))=0$ for $m\ge m_0$, hence $R_1\otimes R_{m-1}\to R_m$ is surjective for $m\ge m_0$. Each $R_m$ is finite-dimensional, so $R$ is generated as a $\bk[R_1]$-module by $\bigoplus_{m<m_0}R_m$. The kernel of $\Sym R_1\to R$ is the ideal of forms vanishing on $\varphi_H(Y)$, since $Y\to\varphi_H(Y)$ is surjective and $\varphi_H(Y)$ is integral; so $\bk[R_1]$ is the homogeneous coordinate ring of the surface $\varphi_H(Y)$, a domain of dimension $3$, and $R$, a domain because $Y$ is integral, is finitely generated of dimension $3$. Finally $R_1R\subseteq\mathfrak m$, $\mathfrak m$ is prime, and a homogeneous $x\in R_+$ has $x^{m_0}\in R_{\ge m_0}\subseteq R_1R$; hence $\sqrt{R_1R}=\mathfrak m$.

(ii) $Y_s=\varphi_H^{-1}(\PP^{h_0-1}\smallsetminus\{s=0\})$ is affine because $\varphi_H$ is finite, and $\Gamma(Y_s,\cF)=(\Gamma_*(\cF)_s)_0$ is \cite[Lem.~II.5.14]{Ha}; the $Y_s$ cover $Y$ since $|H|$ is base-point-free, and the $D_+(s)$ cover $\Proj R$ by (i). Since $Y_s\cap Y_{s'}=Y_{ss'}$ and $D_+(s)\cap D_+(s')=D_+(ss')$, both with ring $(R_{ss'})_0$, the isomorphisms glue. In $R_s$ the element $s$ is a unit of degree $1$, so $(M_s)_t=s^t(M_s)_0$ for every graded $R$-module $M$, and $R_s=(R_s)_0[s,s^{-1}]$ ($s$ is transcendental over $(R_s)_0$ for degree reasons). Hence $\cO_{\Proj R}(t)|_{D_+(s)}$ is free with generator $s^t$ and transition functions $(s/s')^t$, which are those of $\cO_Y(tH)$ trivialized by $s^t$ on $Y_s$. (For $R$ see also \cite[\S5]{GoWa}.)

(iii) By (i), $\Spec R\smallsetminus V(\mathfrak m)=\bigcup_{s\in R_1}\Spec R_s$, and $\Spec R_s=\Spec(R_s)_0[s,s^{-1}]=Y_s\times\mathbb G_m$ is the complement of the zero section in the total space of $\cO_Y(-H)|_{Y_s}$, trivialized by $s$; the gluing is that of $\bigoplus_{t\in\ZZ}\cO_Y(tH)$. As a $\mathbb G_m$-bundle over the smooth surface $Y$, the punctured spectrum $U$ is smooth. Its ring of global functions is $\bigoplus_{t\in\ZZ}H^0(Y,\cO_Y(tH))=R$, because $H^0(Y,\cO_Y(tH))=0$ for $t<0$; and the ring of global functions of a normal integral scheme is integrally closed: an element of $\operatorname{Frac}R=\bk(U)$ integral over $R$ lies in every local ring $\cO_{U,u}$, hence in $\Gamma(U,\cO_U)=R$.
\end{proof}

Throughout, $S$ denotes either the section ring $R$ or, when $H$ is very ample, the homogeneous coordinate ring $A$, and $\mathfrak m=S_+$. A \emph{graded maximal Cohen--Macaulay} (MCM) module over $S$ is a nonzero finitely generated $\ZZ$-graded $S$-module $M$ with $\depth_{\mathfrak m}M=3$, the depth being taken at the irrelevant ideal: $\depth_{\mathfrak m}M\ge j$ means $H^i_{\mathfrak m}(M)=0$ for $i<j$. Since the modules $H^i_{\mathfrak m}(M)$ are $\mathfrak m$-torsion, they are unchanged by localization at $\mathfrak m$, so $\depth_{\mathfrak m}M$ is also the depth of $M_{\mathfrak m}$ over the local ring $S_{\mathfrak m}$.

\begin{theorem}\label{thm:B}
If $H^2<K_Y^2-8\chi(\cO_Y)$, then $R(Y,H)$ admits no graded maximal Cohen--Macaulay module. If moreover $H$ is very ample, the same holds for $A(Y,H)$. Under the same hypothesis neither ring is Cohen--Macaulay: $\depth_{\mathfrak m}R=2$, while $\depth_{A_+}A=2$ if $A=R$ (i.e.\ if $Y\subset\PP^{h_0-1}$ is projectively normal) and $\depth_{A_+}A=1$ otherwise. In both cases the punctured spectrum is smooth, so the vertex is an isolated, non-Cohen--Macaulay singularity.
\end{theorem}

\begin{corollary}\label{cor:C}
Let $Y_n$ $(n\ge3)$ be the Hesse surface of exponent $n$ of Section~\ref{sec:example}: the smooth $(\ZZ/n)^{11}$-cover of the plane blown up at the nine flexes of the Hesse pencil, branched along the twelve lines of the Hesse arrangement and the nine exceptional curves \cite{Hi,BHH}; let $A_n$ be the pull-back of the hyperplane class of its Kummer model $X_n\subset\PP^{11}$, $E_n$ the reduced exceptional divisor, and $H_n=4A_n-E_n$ the divisor of \cite[Thm.~1.2]{An}. Then $|H_n|$ is base-point-free with finite associated morphism (Proposition~\ref{prop:bpf}), and $H_n^2=7n^9<(3n^2-11)n^9=K_{Y_n}^2-8\chi(\cO_{Y_n})=\tau(Y_n)$ (Section~\ref{sec:example}). Consequently the section ring $R(Y_n,H_n)$ is a three-dimensional normal graded $\CC$-domain, with an isolated non-Cohen--Macaulay singularity at the vertex, that admits no nonzero finitely generated graded maximal Cohen--Macaulay module. By \cite[Thm.~1.2]{An} the divisor $H_n$ is moreover very ample; the homogeneous coordinate ring $A(Y_n,H_n)$ of the resulting embedding $Y_n\subset\PP(H^0(\cO(H_n))^\vee)$ likewise admits no nonzero finitely generated graded maximal Cohen--Macaulay module. That is the only point at which \cite{An} is used (Remark~\ref{rem:imported}). No conclusion is drawn here about ungraded maximal Cohen--Macaulay modules over the localization of these rings at the vertex or over its completion.
\end{corollary}

\begin{remark}[provenance]\label{rem:imported}
Theorems~\ref{thm:A} and \ref{thm:B} are proved in Sections~\ref{sec:proofA} and \ref{sec:proofB} for arbitrary $(Y,H)$ as above. The surfaces $Y_n$ are Hirzebruch's \cite{Hi}; the divisor $H_n$ was introduced in \cite{An}. The statements of Corollary~\ref{cor:C} concerning $R(Y_n,H_n)$ use nothing from \cite{An}: the base-point-freeness of $H_n$ and the finiteness of $\varphi_{H_n}$ are proved in Proposition~\ref{prop:bpf} from the structure of the cover, and the invariants $H_n^2$, $K_{Y_n}^2$ and $\chi(\cO_{Y_n})$ are derived in Section~\ref{sec:example} for every $n$ (for $n=3$ all invariants are recomputed). The very ampleness of $H_n$ proved in \cite[Thm.~1.2]{An} enters only the statement about $A(Y_n,H_n)$, which it defines, and is quoted there as a hypothesis.
\end{remark}

\begin{remark}[degree versus signature]\label{rem:signature}
Let $\bk=\CC$. By the classification of surfaces, $\tau(Y)>0$ forces either $Y\cong\PP^2$, where $\tau=\chi(\cO_Y)=1$, or $Y$ of general type: a minimal surface with $\kappa(Y)\le1$ other than $\PP^2$ has $K_Y^2\le8\chi(\cO_Y)$ (ruled over a curve of genus $g$: $K^2=8(1-g)=8\chi$; $\kappa=0,1$: $K^2=0\le8\chi$), and blowing up lowers $K^2$. For $Y$ of general type the Bogomolov--Miyaoka--Yau inequality $K^2\le9\chi$ on the minimal model gives $\tau(Y)\le\chi(\cO_Y)$. Since an ample divisor on $\PP^2$ has $H^2\ge1=\tau(\PP^2)$, the obstruction of Theorem~\ref{thm:A} can only occur on surfaces of general type of positive signature, and it is strongest, for a given $\chi(\cO_Y)$, on ball quotients, where $\tau=\chi(\cO_Y)$. Among the Hesse surfaces, $Y_3$ is a ball quotient (Section~\ref{sec:example}), with $\tau(Y_3)=\chi(\cO_{Y_3})=16\cdot3^9$ and $\deg(Y_3,H_3)=7\cdot3^9$; for $n\ge4$ one has $K_{Y_n}^2-3c_2(Y_n)=-9(n-3)^2n^9<0$, and $\tau(Y_n)/\chi(\cO_{Y_n})=4(3n^2-11)/(21n^2-54n+37)$, whose derivative in $n$ is $-72(n-3)(9n-11)/(21n^2-54n+37)^2$, decreases from $1$ to $4/7$ (Section~\ref{sec:example}), while $\deg(Y_n,H_n)=7n^9<\tau(Y_n)=(3n^2-11)n^9$ for all $n\ge3$. The same reformulation applies to the Ulrich obstruction: an Ulrich bundle is aCM and semistable \cite[Thm.~2.9]{CHGS}, so it forces $\deg\ge\tau$ already by Bogomolov's inequality for the bundle itself (\cite{Be17} for Picard rank one, \cite{An} in general).
\end{remark}

\section{Proof of Theorem \ref{thm:A}}\label{sec:proofA}

\subsection*{The finite projection}
Since $\varphi_H$ is finite, $\varphi_H(Y)$ is a surface in $\PP^{h_0-1}$, so $h_0\ge3$, and a general linear subspace of codimension $3$ (empty if $h_0=3$) is disjoint from it; composing the resulting projection with $\varphi_H$ gives a morphism $\pi\colon Y\to\PP^2$ with $\pi^*\cO_{\PP^2}(1)=\cO_Y(H)$. The ampleness argument of Section~\ref{sec:statements} shows that $\pi$ is finite; moreover $\deg\pi=(\pi^*h)^2=d$. It is flat by miracle flatness \cite[Thm.~23.1]{Mat}, $Y$ being Cohen--Macaulay and $\PP^2$ regular. For a coherent sheaf $\cG$ on $Y$, $\pi_*\cG$ is locally free of rank $rd$ if $\cG$ is locally free of rank $r$, torsion-free if $\cG$ is torsion-free, and
\begin{equation}\label{eq:proj}
H^i\bigl(Y,\cG(t)\bigr)=H^i\bigl(\PP^2,(\pi_*\cG)(t)\bigr)\qquad(i\ge0,\ t\in\ZZ),
\end{equation}
by the projection formula and the vanishing of $R^i\pi_*$ for $i>0$. We write $h$ for the class of a line in $\PP^2$, $\mu_H(\cG)=c_1(\cG)\cdot H/\rk\cG$ and $\mu_h(\cF)=c_1(\cF)\cdot h/\rk\cF$.

\begin{lemma}\label{lem:slope}
For a coherent sheaf $\cG$ of rank $\rho>0$ on $Y$,
\[
c_1(\pi_*\cG)\cdot h=c_1(\cG)\cdot H-\rho\,\frac{k+3d}{2},\qquad\text{i.e.}\qquad
\mu_h(\pi_*\cG)=\frac{\mu_H(\cG)-k/2}{d}-\frac32 .
\]
\end{lemma}

\begin{proof}
Riemann--Roch on $Y$ gives $\chi(\cG(t))=\chi(\cG)+t\,(c_1(\cG)\cdot H-\rho k/2)+t^2\rho d/2$, and on $\PP^2$, $\chi((\pi_*\cG)(t))=\rho d\binom{t+2}{2}+c_1(\pi_*\cG)\cdot h\,(t+\tfrac32)+\mathrm{const}$. The two polynomials coincide by \eqref{eq:proj}; compare the coefficients of $t$.
\end{proof}

\begin{lemma}[Horrocks]\label{lem:horrocks}
A vector bundle $\cE$ on $Y$ is aCM if and only if $\pi_*\cE\cong\bigoplus_{j=1}^{N}\cO_{\PP^2}(a_j)$ for some integers $a_j$, $N=\rk(\cE)\,d$.
\end{lemma}

\begin{proof}
By \eqref{eq:proj}, $\cE$ is aCM iff $H^1_*(\pi_*\cE)=0$, and a vector bundle on $\PP^2$ without intermediate cohomology splits \cite{Ho}.
\end{proof}

\begin{remark}[torsion-free sheaves]\label{rem:tf}
A torsion-free sheaf $\cE$ on $Y$ with $H^1(Y,\cE(t))=0$ for all $t\in\ZZ$ is locally free: in $0\to\cE\to\cE^{\vee\vee}\to\cQ\to0$ the sheaf $\cQ$ has finite length, and $H^0(Y,\cE^{\vee\vee}(t))\to H^0(Y,\cQ(t))$ is surjective for every $t$, so $\cQ=0$ because $H^0(Y,\cE^{\vee\vee}(t))=0$ for $t\ll0$; thus $\cE$ is reflexive, hence locally free on the smooth surface $Y$. Theorem~\ref{thm:A} therefore holds verbatim for torsion-free sheaves, and the bundles of Lemma~\ref{lem:mcm} below could equally be described as torsion-free sheaves.
\end{remark}

\subsection*{A Riemann--Roch identity}
For a torsion-free sheaf $\cF$ of rank $\rho$ on $Y$ put $\nu=c_1(\cF)/\rho\in\NS(Y)_\QQ$, $\mu=\nu\cdot H$, $\delta=\Delta(\cF)/\rho^2$ with $\Delta(\cF)=2\rho c_2(\cF)-(\rho-1)c_1(\cF)^2$, $\beta=\nu-K_Y/2$, and $\alpha=\beta-\frac{\beta\cdot H}{d}H$. Then $\alpha\cdot H=0$, so $\alpha^2\le0$ by the Hodge index theorem, and Riemann--Roch reads
\begin{equation}\label{eq:rr}
\frac{\chi(\cF(t))}{\rho}=\frac d2\Bigl(t+\frac{\mu-k/2}{d}\Bigr)^2-\frac{\sigma}{8}+\frac{\alpha^2-\delta}{2}\qquad(t\in\ZZ).
\end{equation}
Indeed $\chi(\cF(t))/\rho=\chi+\tfrac12\nu\cdot(\nu-K_Y)-\tfrac\delta2+t(\mu-k/2)+t^2d/2$, and $\nu\cdot(\nu-K_Y)=\beta^2-K_Y^2/4$ with $\beta^2=\alpha^2+(\mu-k/2)^2/d$. If $\cF$ is $\mu_H$-semistable, then $\delta\ge0$ by Bogomolov's inequality for torsion-free sheaves in characteristic $0$ \cite[Thm.~3.4.1]{HL}, and \eqref{eq:rr} gives
\begin{equation}\label{eq:ss}
\frac{\chi(\cF(t))}{\rho}\le\frac d2\Bigl(t+\frac{\mu-k/2}{d}\Bigr)^2-\frac{\sigma}{8}.
\end{equation}

\subsection*{Majorization}

\begin{lemma}\label{lem:maj}
Let $W=\bigoplus_{j=1}^N\cO_{\PP^2}(a_j)$ with $a_1\ge\dots\ge a_N$, and let $0=W_0\subset W_1\subset\dots\subset W_s=W$ be a filtration by subsheaves whose quotients $V_i=W_i/W_{i-1}$ have ranks $n_i>0$ and slopes $b_i=\mu_h(V_i)$, where $b_1\ge b_2\ge\dots\ge b_s$. Then, for every real number $c$,
\[
\sum_{i=1}^s\frac{n_i}{N}\,(b_i+c)^2\;\le\;\frac1N\sum_{j=1}^N(a_j+c)^2 .
\]
\end{lemma}

\begin{proof}
Let $T\subset W$ be a subsheaf of rank $p$, and let $\det T:=(\bigwedge^pT)^{\vee\vee}$; it is an invertible sheaf with $c_1(\det T)=c_1(T)$, and it coincides with $\bigwedge^pT$ on the open set $U$ on which $T$ is locally free, whose complement is finite. Since $T\to W$ is generically an injection of vector bundles, $\bigwedge^pT\to\bigwedge^pW=\bigoplus\cO(a_{j_1}+\dots+a_{j_p})$ is nonzero, so some component $\det T|_U\to\cO(a_J)|_U$, $a_J=a_{j_1}+\dots+a_{j_p}$, is nonzero; it extends over $\PP^2$ because a homomorphism between invertible sheaves defined outside a finite subset of the smooth surface $\PP^2$ extends, whence $\deg T:=c_1(T)\cdot h\le a_J\le a_1+\dots+a_p$. Let $x_1\ge\dots\ge x_N$ be the sequence in which each $b_i$ is repeated $n_i$ times, and $\Sigma_x(p)=\sum_{j\le p}x_j$, $\Sigma_a(p)=\sum_{j\le p}a_j$. At $p=n_1+\dots+n_i$ we have $\Sigma_x(p)=\deg W_i\le\Sigma_a(p)$; between two such values $\Sigma_x$ is affine and $\Sigma_a$ is concave, hence $\Sigma_x(p)\le\Sigma_a(p)$ for all $p$, with $\Sigma_x(N)=\Sigma_a(N)=\deg W$. Thus $x$ is majorized by $a$, and Karamata's inequality for the convex function $y\mapsto(y+c)^2$ gives $\sum_j(x_j+c)^2\le\sum_j(a_j+c)^2$, which is the claim. (Explicitly, with $S_p=\Sigma_a(p)-\Sigma_x(p)\ge0$, $S_N=0$: $\sum_j\bigl[(a_j+c)^2-(x_j+c)^2\bigr]=\sum_{p<N}S_p\bigl[(a_p+x_p)-(a_{p+1}+x_{p+1})\bigr]\ge0$.) The inequality $\Sigma_x\le\Sigma_a$ is a special case of Shatz's theorem that the Harder--Narasimhan polygon of a sheaf lies above the polygon of any filtration by subsheaves \cite{Sh}, the Harder--Narasimhan polygon of $W$ being given by the $a_j$; the passage from majorization to the inequality of sums is the Hardy--Littlewood--P\'olya--Karamata inequality \cite{HLP,Ka}. The lemma is thus a restatement of \cite[Lem.~1.5]{Lan}, where the same inequality is proved for two convex polygons with common endpoints, one contained in the other, and is used in the same structural place, to transport an estimate from one Harder--Narasimhan polygon to another; the shift by $c$ is immaterial, the polygons sharing their endpoints, so that $\sum_jx_j=\sum_ja_j$ and the terms linear in $c$ cancel.
\end{proof}

\begin{remark}[what the majorization step replaces]\label{rem:whymaj}
The push-forward $\pi_*$ does not preserve $\mu$-semistability, so the filtration $\pi_*\cE_\bullet$ of $W=\pi_*\cE$ used below is in general \emph{not} the Harder--Narasimhan filtration of $W$, and its slopes $b_1>\dots>b_s$ bear no a priori relation to the splitting type $a_1\ge\dots\ge a_N$ of $W$. Lemma~\ref{lem:maj} supplies the only relation that is needed, and it uses nothing about the filtration beyond three facts: that each $\pi_*\cE_i$ is a subsheaf of $W$, that its rank is $d\cdot\rk\cE_i$, and that the slopes $b_i$ are decreasing. The last is essential, and is not a formal consequence of the other two: on $\PP^2$, for two forms of degree $M>0$ without common factor and with intersection $Z$, the extension $0\to\cO(-M)\to\cO^{\oplus2}\to I_Z(M)\to0$ filters $W=\cO^{\oplus2}$, of splitting type $(0,0)$, by subsheaves with rank-one quotients of slopes $(-M,M)$, and $\tfrac12\bigl((-M)^2+M^2\bigr)=M^2>0$ violates the conclusion of the lemma. Here the monotonicity is supplied by Lemma~\ref{lem:slope}, the map $\mu\mapsto(\mu-k/2)/d-\tfrac32$ being strictly increasing, so that $\mu_1>\dots>\mu_s$ forces $b_1>\dots>b_s$. With it, the resulting inequality between the second moments of the two sequences is exactly what the Riemann--Roch comparison of the next subsection consumes. No semistability statement about $W$, and no compatibility between the two filtrations, is claimed or used.

It also explains why the conclusion consumes the exact value of $\chi(\cE)$ and not merely the inequality $\chi(\cE(t))=h^0+h^2\ge0$, which is what suffices for line bundles and for semistable bundles (Remark~\ref{rem:rr}). There, one evaluates at the integer nearest the vertex of a single parabola; here the parabolas attached to the several Harder--Narasimhan factors have different vertices, and $\chi(\cE(t))\ge0$ constrains only their sum. Summing \eqref{eq:ss} at a common $t=0$ gives
\[
0\;\le\;\frac{\chi(\cE)}{r}\;\le\;\frac d2\Bigl(v_b+\bigl(\bar a+\tfrac32\bigr)^2\Bigr)-\frac\sigma8
\]
in the notation of Proposition~\ref{prop:identity}, and neither the weighted variance $v_b$ of the slopes $b_i$ nor their mean $\bar a$ is controlled by $d$ and $\sigma$, so this yields nothing. Computing $\chi(\cE)$ exactly through the Horrocks splitting produces the term $\bigl(\bar a+\tfrac32\bigr)^2$ a second time, with the opposite sign; it cancels, and what remains to be proved is precisely $v_a\ge v_b$, which is Lemma~\ref{lem:maj}.
\end{remark}

\subsection*{Conclusion}
Let $\cE$ be a nonzero aCM bundle of rank $r$, $N=rd$, and $W=\pi_*\cE\cong\bigoplus_{j=1}^N\cO(a_j)$ with $a_1\ge\dots\ge a_N$ (Lemma~\ref{lem:horrocks}). Let $0=\cE_0\subset\cE_1\subset\dots\subset\cE_s=\cE$ be the Harder--Narasimhan filtration with respect to $\mu_H$; its quotients $\cQ_i$ are torsion-free, $\mu_H$-semistable, of ranks $r_i$ and slopes $\mu_1>\dots>\mu_s$. Since $\pi_*$ is exact and preserves torsion-freeness, $W_i:=\pi_*\cE_i$ is a filtration of $W$ -- in general not its Harder--Narasimhan filtration (Remark~\ref{rem:whymaj}) -- with torsion-free quotients $\pi_*\cQ_i$ of ranks $n_i=dr_i$ and, by Lemma~\ref{lem:slope}, slopes
\[
b_i=\frac{\mu_i-k/2}{d}-\frac32,\qquad b_1>\dots>b_s .
\]
Put $w_i=r_i/r=n_i/N$. Additivity of $\chi$, \eqref{eq:ss} at $t=0$ for each $\cQ_i$, and Lemma~\ref{lem:maj} with $c=\tfrac32$ give
\[
\frac{\chi(\cE)}{r}=\sum_iw_i\frac{\chi(\cQ_i)}{r_i}
\le\frac d2\sum_iw_i\bigl(b_i+\tfrac32\bigr)^2-\frac\sigma8
\le\frac{d}{2N}\sum_j\bigl(a_j+\tfrac32\bigr)^2-\frac\sigma8 .
\]
On the other hand, by \eqref{eq:proj} and $\chi(\cO_{\PP^2}(a))=\tfrac12(a+1)(a+2)=\tfrac12\bigl[(a+\tfrac32)^2-\tfrac14\bigr]$,
\[
\frac{\chi(\cE)}{r}=\frac1r\sum_j\chi\bigl(\cO_{\PP^2}(a_j)\bigr)=\frac{d}{2N}\sum_j\bigl(a_j+\tfrac32\bigr)^2-\frac d8 .
\]
Comparing the two displays gives $d\ge\sigma$. Hence if $d<\sigma$ no nonzero aCM bundle exists. Note that of the aCM hypothesis only the splitting of $\pi_*\cE$ has been used, not the nonnegativity of $\chi(\cE(t))$. \qed

\begin{proposition}[the defect identity]\label{prop:identity}
Keeping all terms, the argument proves, for every nonzero aCM bundle $\cE$, with the notation above,
\[
\frac{d-\sigma}{4}=d\,(v_a-v_b)+\sum_iw_i\bigl(\delta(\cQ_i)-\alpha(\cQ_i)^2\bigr),
\qquad v_a=\tfrac1N\textstyle\sum_j(a_j-\bar a)^2,\quad v_b=\sum_iw_i(b_i-\bar a)^2,
\]
where $\bar a=\frac1N\sum_ja_j=\sum_iw_ib_i$, and every term on the right is $\ge0$: $v_a\ge v_b$ is Lemma~\ref{lem:maj} with $c=-\bar a$, $\delta(\cQ_i)\ge0$ is Bogomolov's inequality and $\alpha(\cQ_i)^2\le0$ is the Hodge index theorem. The identity distributes the defect $d-\sigma$ among the gap between the second moments of the two polygons, the discriminants of the Harder--Narasimhan factors, and the squares $\alpha(\cQ_i)^2$ measuring the deviation of their normalized first Chern classes, shifted by $K_Y/2$, from the line spanned by $H$. The term $v_a-v_b$ should not be read as a measure of the instability of $\cE$: it is already positive for $\cE=\cO$ on $(\PP^2,\cO(2))$, where $v_b=0$ and $v_a=\tfrac3{16}$ (Remark~\ref{rem:checks}). The argument needs neither that the $\cQ_i$ be aCM nor that $\pi_*$ preserve semistability.
\end{proposition}

\begin{remark}[numerical checks]\label{rem:checks}
As a check of the signs, the identity of Proposition~\ref{prop:identity} was evaluated on unstable aCM bundles: $\cO\oplus\cO(1)$ on a smooth quartic surface in $\PP^3$, where $\pi_*\cO=\bigoplus_{i=0}^3\cO(-i)$, $a=(1,0,0,-1,-1,-2,-2,-3)$, $v_a=\tfrac32$, $b=(-\tfrac12,-\tfrac32)$, $v_b=\tfrac14$ ($5=5$); $\cO(1,0)\oplus\cO$ on $\PP^1\times\PP^1$ with $H=\cO(1,1)$, where $\pi_*\cO(1,0)=\cO^2$, $\pi_*\cO=\cO\oplus\cO(-1)$ and $\alpha(\cO(1,0))^2=-\tfrac12$ ($\tfrac12=\tfrac14+\tfrac14$); and $\cO(\Theta)\oplus\cO(2\Theta)$ on a principally polarized abelian surface with $H=3\Theta$, an irregular surface with $K=0$, $\sigma=0$, $d=18$ (the factor $3$ in $H=3\Theta$ is what makes $\cE$ aCM: $H^1(\cO(u\Theta))\ne0$ only for $u=0$, and $m+3t\ne0$ for $m=1,2$ and $t\in\ZZ$): here $h^0(\cO(m\Theta))=m^2$ and $H^i(\cO(m\Theta))=0$ for $i>0$, $m>0$, so the splitting types follow from $h^0(\pi_*\cO(\Theta)(t))=(3t+1)^2$ and $h^0(\pi_*\cO(2\Theta)(t))=(3t+2)^2$ for $t\ge0$ (and $0$ for $t\le-1$), namely $\pi_*\cO(\Theta)=\cO\oplus\cO(-1)^{13}\oplus\cO(-2)^4$ and $\pi_*\cO(2\Theta)=\cO^4\oplus\cO(-1)^{13}\oplus\cO(-2)$, whence $\bar a=-1$, $v_a=\tfrac{10}{36}=\tfrac5{18}$, $b=(-\tfrac56,-\tfrac76)$, $v_b=\tfrac1{36}$ and $\alpha=0$ ($\tfrac92=\tfrac92$). A semistable check with $d\ne1$: $\cO$ on $(\PP^2,\cO(2))$, where $\pi_*\cO=\cO\oplus\cO(-1)^3$, $v_a=\tfrac3{16}$, $v_b=\delta=\alpha=0$ and $\tfrac{4-1}4=4\cdot\tfrac3{16}$.
\end{remark}

\begin{remark}[semistable bundles and line bundles]\label{rem:rr}
For every aCM bundle $\cE$, semistable or not, evaluating the identity \eqref{eq:rr} at the integer nearest to the vertex and using $\chi(\cE(t))=h^0+h^2\ge0$ gives $\delta(\cE)\le\tfrac{d-\sigma}4+\alpha(\cE)^2\le\tfrac{d-\sigma}4$, i.e.\ $\Delta(\cE)\le\tfrac{r^2}{4}(d-\sigma)$. If $\cE$ is $\mu_H$-semistable this contradicts Bogomolov's inequality when $d<\sigma$, and the Harder--Narasimhan filtration is not needed: semistability enters only through Bogomolov's inequality, and Section~\ref{sec:proofA} adds the treatment of unstable bundles. For a line bundle $\delta=0$, and this is the elementary observation that, when $\sigma>d$ (so that $\sigma>0$), $\chi(\cE(t))$ is negative on an open interval of length $\sqrt{(\sigma-4\alpha^2)/d}\ge\sqrt{\sigma/d}>1$, which therefore contains an integer. Nothing in this paragraph uses the base-point-freeness of $H$: \eqref{eq:rr}, the Hodge index theorem and Bogomolov's inequality require only that $H$ be ample, so for $H$ merely ample and $\cE$ semistable of any rank one still obtains $d\ge\sigma$. Base-point-freeness is used only for bundles carrying no semistability assumption, through the finite projection to $\PP^2$; for $H$ merely ample the argument applied to a base-point-free multiple $mH$ gives only $m^2d\ge\sigma$ in rank $\ge2$, and whether $d\ge\sigma$ holds for ample $H$ in rank $\ge2$ \emph{without} a semistability assumption is not addressed here. The semistable case is thus the Bogomolov obstruction, used in \cite{An} for Ulrich bundles; the content of Theorem~\ref{thm:A} is the unstable case, where the Harder--Narasimhan filtration of $\cE$ is compared with the Horrocks splitting of $\pi_*\cE$ through the majorization Lemma~\ref{lem:maj}.
\end{remark}

\section{Proof of Theorem \ref{thm:B}}\label{sec:proofB}

Recall from Section~\ref{sec:statements} that $S$ denotes either the section ring $R=R(Y,H)$ or, when $H$ is very ample, the homogeneous coordinate ring $A=A(Y,H)$, and $\mathfrak m=S_+$. In both cases
\begin{equation}\label{eq:cover}
\sqrt{S_1S}=\mathfrak m :
\end{equation}
for $R$ this is Lemma~\ref{lem:ring}(i), and for $A$ it is clear, $A_+=A_1A$ being generated in degree $1$. Lemma~\ref{lem:ring}(ii) is stated for $R$, but holds verbatim for $A$: by Serre vanishing on $\PP^{h_0-1}$ one has $A_t=H^0(Y,\cO_Y(tH))=R_t$ for $t\gg0$, so the graded module $R/A$ has finite length and is killed by a power of any $s\in A_1=R_1$; hence $A_s=R_s$, and the charts $D_+(s)$ of $\Proj A$ coincide with those of $\Proj R$, together with their twisted sheaves. By \eqref{eq:cover} and Lemma~\ref{lem:ring}(ii), $Y=\Proj S$ is covered by the affine opens $D_+(s)$, $s\in S_1$, on which $S_s\cong(S_s)_0[s,s^{-1}]$ and $(M_s)_t=s^t(M_s)_0$ for every graded $S$-module $M$; hence $\cO_{\Proj S}(t)=\cO_{\Proj S}(1)^{\otimes t}=\cO_Y(tH)$ for all $t\in\ZZ$ and $\widetilde{M(t)}=\widetilde M\otimes\cO_Y(tH)$.

\begin{remark}[a warning]\label{rem:warning}
These properties do not follow from the weaker requirements that $\Proj S=Y$, $\cO_{\Proj S}(1)\cong\cO_Y(H)$ and $S_1$ generate $\cO_Y(H)$, and Lemma~\ref{lem:mcm} below fails without them: for $S=\CC[u,v,w]^{(3)}$ with $\deg u=1$, $\deg v=\deg w=2$, one has $\Proj S=\PP^2$, $\cO_{\Proj S}(1)=\cO_{\PP^2}(1)$ generated by $S_1=u\langle u^2,v,w\rangle$, but $S_1\subset(u)\cap S\subsetneq S_+$, $\cO_{\Proj S}(2)=\cO_{\PP^2}(3)$, and $S$, which is free over $\CC[u^6,v^3,w^3]$, is a graded maximal Cohen--Macaulay module over itself with $\dim S_2=10\ne6=h^0(\cO_{\PP^2}(2))$. The failure is therefore explicit: $\widetilde S=\cO_{\PP^2}$ is an aCM bundle, but already as graded vector spaces $S\not\cong\bigoplus_{t\in\ZZ}H^0(\PP^2,\cO_{\PP^2}(t))$. Without the identity $\cO_{\Proj S}(t)=\cO_{\Proj S}(1)^{\otimes t}$ of Lemma~\ref{lem:ring}(ii), which fails here at $t=2$, this sum formed with ordinary tensor powers does not automatically carry the graded $S$-module structure used in Lemma~\ref{lem:mcm}; the two constructions cannot be identified as in that lemma. Condition \eqref{eq:cover} says that the punctured spectrum is covered by the $\mathbb G_m$-stable affine charts $D(s)=\Spec(S_s)_0[s,s^{-1}]$, $s\in S_1$, on which the action is trivialized; for the ring above this fails along the line $u=0$.
\end{remark}

The correspondence that follows is classical when $H$ is very ample. The argument given here is valid in any characteristic and uses \eqref{eq:cover}, which by Remark~\ref{rem:warning} cannot simply be dropped. It should not be confused with \cite[Cor.~4.5(i)]{ShTa}, which, for an ample $\QQ$-divisor $D$ on a normal projective variety $X$ of dimension $\ge3$ over an $F$-finite field of characteristic $p$, with $\cO_X(\lfloor iD\rfloor)$ invertible for all $i$, characterizes the existence of a graded maximal Cohen--Macaulay module over $R(X,D)$ by that of a coherent sheaf on $X$ whose stalks are maximal Cohen--Macaulay. The latter is a purely local condition, satisfied by $\cO_X$ on every smooth $X$, and strictly weaker than being aCM. Of the two implications, the one going from a graded maximal Cohen--Macaulay module to such a sheaf is characteristic-free; the converse is proved there by passing to Veronese direct summands of Frobenius push-forwards \cite[Prop.~4.1]{ShTa}, while the characteristic-free part of that argument \cite[Lem.~4.4]{ShTa} yields only a generalized Cohen--Macaulay module. It is that converse which fails in characteristic zero: Corollary~\ref{cor:shta} below exhibits pairs $(X,D)$ satisfying every hypothesis of \cite[Cor.~4.5(i)]{ShTa} once its positive-characteristic, $F$-finite base field is replaced by an algebraically closed field of characteristic zero -- the $F$-finiteness being part of that hypothesis and meaningless without it -- and for which it is false.

\begin{lemma}\label{lem:mcm}
$M\mapsto\widetilde M$ and $\cE\mapsto\Gamma_*(\cE)=\bigoplus_{t\in\ZZ}H^0(Y,\cE(t))$ induce mutually inverse bijections, up to isomorphism, between graded MCM $S$-modules and aCM vector bundles of positive rank on $(Y,H)$.
\end{lemma}

\begin{proof}
Let $M$ be a finitely generated graded $S$-module. By \eqref{eq:cover} there are finitely many elements of $S_1$ generating an $\mathfrak m$-primary ideal; computing local cohomology with the \v Cech complex on them, and comparing with the \v Cech complex of $\widetilde M(t)=\widetilde M\otimes\cO_Y(tH)$ on the corresponding cover of $Y$, one gets the exact sequence and isomorphisms (cf.\ \cite[(2.1.5)]{EGA3}, \cite[(5.1.6)]{GoWa}, \cite[20.4.4]{BS})
\begin{equation}\label{eq:lc}
0\to H^0_{\mathfrak m}(M)\to M\to\Gamma_*(\widetilde M)\to H^1_{\mathfrak m}(M)\to0,\qquad
H^{i+1}_{\mathfrak m}(M)_t\cong H^i\bigl(Y,\widetilde M(t)\bigr)\ \ (i\ge1,\ t\in\ZZ).
\end{equation}
Suppose $\depth_{\mathfrak m}M=3$. Then $H^i_{\mathfrak m}(M)=0$ for $i\le2$, so $M\cong\Gamma_*(\widetilde M)$ and $H^1(Y,\widetilde M(t))=0$ for all $t$. As $3=\depth_{\mathfrak m}M\le\dim M\le\dim S=3$ and $S$ is a domain, $\Supp M=\Spec S$. It remains to see that $\widetilde M$ is locally free. Fix a closed point $y\in Y$ and, by \eqref{eq:cover}, $s\in S_1$ with $y\in D_+(s)$. With $B_0:=(S_s)_0$ one has $D_+(s)=\Spec B_0$ and
\[
S_s\cong B_0[T,T^{-1}]\ (T\mapsto s),\qquad M_s\cong(M_s)_0\otimes_{B_0}B_0[T,T^{-1}].
\]
Let $\mathfrak q\subset B_0$ be the ideal of $y$, $B=(B_0)_{\mathfrak q}=\cO_{Y,y}$, and $\mathfrak p\subset S$ the homogeneous prime of $y$. Then
\[
S_{\mathfrak p}\cong B[T,T^{-1}]_{\mathfrak qB[T,T^{-1}]},\qquad M_{\mathfrak p}\cong\widetilde M_y\otimes_BS_{\mathfrak p},
\]
and $B\to S_{\mathfrak p}$ is a local, flat, hence faithfully flat, extension of regular local rings of dimension $2$ (the closed fibre is the field $\kappa(\mathfrak q)(T)$). The module $M_{\mathfrak p}$ is a localization of the maximal Cohen--Macaulay module $M_{\mathfrak m}$ (Section~\ref{sec:statements}); the Cohen--Macaulay property localizes \cite[\S2.1]{BH}, and full support makes $M_{\mathfrak p}$ maximal Cohen--Macaulay over the regular ring $S_{\mathfrak p}$, hence free (Auslander--Buchsbaum). By faithfully flat descent of flatness, the finite $B$-module $\widetilde M_y$ is flat, hence free. So $\widetilde M$ is an aCM bundle, of positive rank since $\Supp M=\Spec S$.

Conversely, let $\cE$ be an aCM bundle. The graded $R$-module $M:=\Gamma_*(\cE)$ is finitely generated: $H^0(Y,\cE(t))=0$ for $t\ll0$ by Serre duality and Serre vanishing, each $H^0(Y,\cE(t))$ is finite-dimensional, and $R_1\cdot H^0(Y,\cE(t))=H^0(Y,\cE(t+1))$ for $t\gg0$. For the last point let $V=R_1$; the Koszul complex $0\to\bigwedge^{h_0}V\otimes\cO_Y(-h_0H)\to\dots\to V\otimes\cO_Y(-H)\to\cO_Y\to0$ of a basis of $V$ is exact because $|H|$ is base-point-free, and after tensoring with $\cE((t+1)H)$ the vanishing of $H^i\bigl(Y,\bigwedge^{i+1}V\otimes\cE((t-i)H)\bigr)$ for $i=1,2$ and $t\gg0$ (Serre vanishing) gives the surjectivity of $V\otimes H^0(Y,\cE(t))\to H^0(Y,\cE(t+1))$. When $S=A$, $R$ is a finite $A$-module by Lemma~\ref{lem:ring}(i), $A=\bk[R_1]$ being the subring generated by $R_1$; so $M$ is a finitely generated graded $S$-module in both cases, and $\widetilde M=\cE$ because $(M_s)_0=\Gamma(D_+(s),\cE)$ for $s\in S_1$ (Lemma~\ref{lem:ring}(ii)). By \eqref{eq:lc}, $H^0_{\mathfrak m}(M)=H^1_{\mathfrak m}(M)=0$ because $M=\Gamma_*(\widetilde M)$, and $H^2_{\mathfrak m}(M)=\bigoplus_tH^1(Y,\cE(t))=0$; so $\depth_{\mathfrak m}M=3$. The two constructions are inverse to each other by \eqref{eq:lc}.
\end{proof}

\begin{proof}[Proof of Theorem \ref{thm:B}]
The statements about $R$ and $A$ follow from Lemma~\ref{lem:mcm} and Theorem~\ref{thm:A}. For the vertex: by Lemma~\ref{lem:ring}(iii) the punctured spectrum of $R$ is the complement of the zero section in $\Tot(\cO_Y(-H))$, hence smooth. Under $d<\sigma$, Theorem~\ref{thm:A} applied to $\cO_Y$ (or Remark~\ref{rem:rr}) gives $t\in\ZZ$ with $H^1(Y,\cO_Y(tH))\ne0$, i.e.\ $H^2_{\mathfrak m}(R)\ne0$ by \eqref{eq:lc}; normality of $R$ gives $\depth_{\mathfrak m}R\ge2$, hence $\depth_{\mathfrak m}R=2<3=\dim R$. Thus $R$ is not Cohen--Macaulay and the vertex is a singular point, isolated by the first part. For $A$: as noted before Lemma~\ref{lem:mcm}, $A_s=R_s$ for every $s\in A_1$, and the same charts identify $\Spec A\smallsetminus\{A_+\}$ with the complement of the zero section in $\Tot(\cO_Y(-H))$; the punctured spectrum is smooth. Moreover $H^0_{A_+}(A)=0$ since $A$ is a domain of positive dimension, $H^1_{A_+}(A)\cong R/A$ by \eqref{eq:lc} applied to $M=A$ (for which $\widetilde M=\cO_Y$ and $\Gamma_*(\cO_Y)=\bigoplus_tH^0(Y,\cO_Y(tH))=R$), and $H^2_{A_+}(A)_t\cong H^1(Y,\cO_Y(tH))\ne0$ for the $t$ above. Hence $\depth_{A_+}A=1$ if $A\ne R$ and $\depth_{A_+}A=2$ if $A=R$; in either case $A$ is not Cohen--Macaulay and its vertex is an isolated singularity.
\end{proof}

\begin{remark}
Without the hypothesis $d<\sigma$ the vertex need not be singular: for $(Y,H)=(\PP^2,\cO(1))$ the section ring is the polynomial ring $\bk[x_0,x_1,x_2]$. The smoothness of the punctured spectrum holds in general (Lemma~\ref{lem:ring}(iii)).
\end{remark}

\section{The example}\label{sec:example}

\subsection*{The Hesse surfaces}
Let $\ell_0,\dots,\ell_{11}$ be the twelve lines of the Hesse arrangement in $\PP^2$, the lines through pairs of the nine flexes of a smooth plane cubic; the arrangement has nine quadruple points, the flexes, each on four lines, and twelve double points, and the twelve lines are the four triangles of the Hesse pencil. Let $\beta\colon Z:=\mathrm{Bl}_9\PP^2\to\PP^2$ be the blow-up of the nine flexes, with exceptional curves $E_p$, let $L=\beta^*\cO(1)$, and let $\tilde L_i=L-\sum_{p\in\ell_i}E_p$ be the strict transforms. Following Hirzebruch \cite{Hi} (see also \cite{BHH}), let $X_n\subset\PP^{11}$, $n\ge2$, be the Kummer cover of $\PP^2$ of exponent $n$ branched along the arrangement: choosing $\ell_0,\ell_1,\ell_2$ to be the coordinate lines, $X_n$ is the complete intersection of the nine hypersurfaces $x_i^n=\ell_i(x_0^n,x_1^n,x_2^n)$, $i=3,\dots,11$, of degree $n$, and $p\colon X_n\to\PP^2$, $[x]\mapsto[x_0^n:x_1^n:x_2^n]$, is a $(\ZZ/n)^{11}$-cover of degree $n^{11}$ with $p^*\cO_{\PP^2}(1)=\cO_{X_n}(n)$. Let $f\colon Y_n\to Z$ be the normalization of $Z\times_{\PP^2}X_n$: it is the $(\ZZ/n)^{11}$-cover of $Z$ branched along $B=\sum_i\tilde L_i+\sum_pE_p$, a simple normal crossing divisor, with ramification index $n$ along every component \cite{Hi,Pa}. Write $G=(\ZZ/n)^{12}/\langle(1,\dots,1)\rangle$, and let $e_i$ be the images of the standard basis vectors. For each flex $p$, let $T_p$ be the set of the four lines through $p$. The inertia group of $\tilde L_i$ is generated by $e_i$ and that of $E_p$ by $\sum_{j\in T_p}e_j$; at every node of $B$ the two inertia groups generate a subgroup $(\ZZ/n)^2$, so the cover is \'etale locally $(u,v)\mapsto(u^n,v^n)$ and $Y_n$ is smooth \cite{Pa}.
The induced morphism $r\colon Y_n\to X_n$, with $\beta\circ f=p\circ r$, is a resolution of the singularities of $X_n$. These lie over the nine flexes: over a double point of the arrangement $X_n$ is \'etale locally $\{x_i^n=u,\ x_j^n=v\}$ with $u,v$ local coordinates, over a smooth point of a line it is $\{x_i^n=u\}$, and elsewhere it is \'etale over $\PP^2$; so $r$ is an isomorphism over $p^{-1}(\PP^2\smallsetminus\{\text{flexes}\})$. Put
\[
A_n:=r^*\cO_{X_n}(1),\qquad E'_p:=f^{-1}(E_p)_{\mathrm{red}},\qquad E_n:=\sum_pE'_p,\qquad D_i:=f^{-1}(\tilde L_i)_{\mathrm{red}} .
\]
Total ramification along $B$ gives $f^*E_p=nE'_p$ and $f^*\tilde L_i=nD_i$, and $f^*L=r^*p^*\cO(1)=nA_n$; that is, $A_n=\tfrac1nf^*L$ and $E_n=\tfrac1nf^*\sum_pE_p$, with $E_n$ the reduced exceptional divisor of $r$. Following \cite[Thm.~1.2]{An} put
\[
H_n:=4A_n-E_n,\qquad\text{so that}\qquad f^*H_Z=nH_n\quad\text{for}\quad H_Z:=4L-\textstyle\sum_pE_p .
\]

\begin{proposition}\label{prop:bpf}
For every $n\ge2$ the linear system $|H_n|$ is base-point-free and $H_n\cdot C>0$ for every irreducible curve $C\subset Y_n$; hence $\varphi_{H_n}$ is finite and $H_n$ is ample.
\end{proposition}

\begin{proof}
(a) $A_n=r^*\cO_{X_n}(1)$ is base-point-free.

(b) The coordinate $x_i$ of $\PP^{11}$ restricts to a section of $\cO_{X_n}(1)$ with $n\operatorname{div}(x_i)=\operatorname{div}(x_i^n)=p^*\ell_i$; hence $r^*x_i\in H^0(Y_n,\cO(A_n))$ satisfies $n\operatorname{div}(r^*x_i)=f^*\beta^*\ell_i=f^*\bigl(\tilde L_i+\sum_{p\in\ell_i}E_p\bigr)=n\bigl(D_i+\sum_{p\in\ell_i}E'_p\bigr)$, i.e.
\[
\operatorname{div}(r^*x_i)=D_i+\sum_{p\in\ell_i}E'_p .
\]

(c) Let $T$ be one of the four triangles of the Hesse pencil. Its three lines contain the nine flexes, each flex lying on exactly one of them (the vertices of the four triangles are the twelve double points). By (b), $\Phi_T:=\prod_{i\in T}r^*x_i\in H^0(Y_n,\cO(3A_n))$ has divisor $\sum_{i\in T}D_i+E_n$; hence $\Phi_T=\Psi_T\cdot z$, where $z$ is the canonical section of $\cO(E_n)$ and $\Psi_T\in H^0(Y_n,\cO(3A_n-E_n))$ has divisor $\sum_{i\in T}D_i$.

(d) Two lines $\ell_i\in T$ and $\ell_j\in T'$ from different triangles meet at a flex: $\ell_i$ carries three flexes, each lying on exactly one line of $T'$, and these three lines are distinct (a line of $T'$ through two flexes of $\ell_i$ would be $\ell_i$); so every line of $T'$ meets $\ell_i$ at one of its flexes. Since distinct lines through a blown-up point have disjoint strict transforms, $\tilde L_i\cap\tilde L_j=\emptyset$ and $D_i\cap D_j=f^{-1}(\tilde L_i\cap\tilde L_j)=\emptyset$. Hence for $T\ne T'$ the sections $\Psi_T$ and $\Psi_{T'}$ have disjoint zero loci, and $|3A_n-E_n|$ is base-point-free.

(e) $H_n=A_n+(3A_n-E_n)$ is base-point-free, as a sum of base-point-free divisors. It is ample: on $Z$ one has $H_Z=4L-\sum_pE_p=-K_Z+L$, where $-K_Z=3L-\sum_pE_p$, the class of the Hesse pencil, is nef because the pencil has no base points on $Z$, and $L$ is nef; so $H_Z\cdot E_p=1$, $H_Z\cdot C\ge L\cdot C>0$ for every other irreducible curve $C$, and $H_Z^2=7>0$, whence $H_Z$ is ample by the Nakai--Moishezon criterion, and so is $H_n=\tfrac1nf^*H_Z$, $f$ being finite. An ample base-point-free divisor defines a finite morphism (Section~\ref{sec:statements}).
\end{proof}

Thus the pairs $(Y_n,H_n)$ satisfy the standing hypotheses of Section~\ref{sec:statements}. Moreover, by the formula for $K_{Y_n}$ below, $K_{Y_n}=(2n-3)H_n+nA_n$ is ample for every $n\ge2$, $H_n$ being ample and $A_n$ nef; so the surfaces $Y_n$ are minimal of general type. In particular, $r\colon Y_n\to X_n$ has no exceptional $(-1)$-curves and is the minimal resolution. By uniqueness, it identifies $Y_n$ over $X_n$ with the blow-up model of \cite[\S3.2]{An}, preserving $A_n=r^*\cO_{X_n}(1)$ and the reduced exceptional divisor $E_n$. By \eqref{eq:alln} below, $H_n^2=7n^9<(3n^2-11)n^9=K_{Y_n}^2-8\chi(\cO_{Y_n})$ exactly for $n\ge3$. This is the hypothesis of Theorem~\ref{thm:B}, and Corollary~\ref{cor:C} follows. (By \cite[Thm.~1.2]{An}, $H_n$ is moreover very ample for $n\ge3$; this is not used here.)

For $n=3$, with $\lambda:=3^9=19683$, the invariants are
\begin{gather*}
A^2=\lambda,\qquad E^2=-9\lambda,\qquad A\cdot E=0,\qquad K=15A-3E,\\
H^2=7\lambda,\qquad K\cdot H=33\lambda,\qquad K^2=144\lambda,\qquad \chi(\cO_Y)=16\lambda,\qquad c_2=48\lambda,
\end{gather*}
so $Y_3$, whose canonical class is ample and satisfies $K^2=3c_2$, is a compact ball quotient by Yau's theorem (Hirzebruch's example \cite{Hi}), $\sigma=16\lambda$ and $d=7\lambda$; moreover $q(Y_3)=154$ and $p_g(Y_3)=315081$. The rest of this section derives these numbers, and the invariants of $Y_n$ for every $n$, from the structure of the cover.

\subsection*{Eigensheaves}
By Pardini \cite{Pa}, $f_*\cO_{Y_n}=\bigoplus_aL_a^{-1}$ over the characters $a\in\{0,\dots,n-1\}^{12}$ with $\sum_ia_i\equiv0\pmod n$, where
\[
L_a=d_a\,L-\sum_pb_p(a)\,E_p,\qquad d_a=\frac1n\sum_ia_i,\qquad b_p(a)=\Bigl\lfloor\frac1n\sum_{i\in T_p}a_i\Bigr\rfloor,
\]
$T_p$ being the set of the four lines through $p$; the residue of $\sum_{i\in T_p}a_i$ modulo $n$ is the character of the inertia along $E_p$, and subtracting it produces the integer part. Consequently, for every divisor $D_Z$ on $Z$,
\begin{equation}\label{eq:eigen}
h^i\bigl(Y_n,f^*D_Z\bigr)=\sum_ah^i\bigl(Z,D_Z-L_a\bigr)\qquad(i\ge0).
\end{equation}
The projection formula, applied to $nA_n=f^*L$ and $nE_n=f^*\sum_pE_p$ with $\deg f=n^{11}$, gives
\begin{equation}\label{eq:hesse-intersections}
\begin{gathered}
A_n^2=\frac{n^{11}}{n^2}L^2=n^9,\qquad
A_n\cdot E_n=\frac{n^{11}}{n^2}L\cdot\sum_pE_p=0,\\
E_n^2=\frac{n^{11}}{n^2}\Bigl(\sum_pE_p\Bigr)^2=-9n^9.
\end{gathered}
\end{equation}

\subsection*{The canonical class}
The branch divisor is reduced, with ramification index $n$ along each component, and has class $B=\sum_i\tilde L_i+\sum_pE_p=12L-3\sum_pE_p$. With $K_Z=-3L+\sum_pE_p$, Hurwitz gives
\[
K_{Y_n}=f^*\Bigl(K_Z+\bigl(1-\tfrac1n\bigr)B\Bigr)=(9n-12)A_n+(3-2n)E_n ,
\]
whence, for every $n$,
\[
K_{Y_n}^2=(45n^2-108n+63)\,n^9=9(5n-7)(n-1)\,n^9,\qquad K_{Y_n}\cdot H_n=(18n-21)\,n^9,\qquad H_n^2=7n^9 ;
\]
for $n=3$, $K=15A-3E$, $K^2=144\lambda$, $K\cdot H=33\lambda$, $H^2=7\lambda$. The Euler polynomial in the direction $H$ alone would determine $H^2$, $K\cdot H$ and $\chi$ but not $K^2$; this is why $K^2$ is derived from the cover.

\subsection*{Euler characteristics for every \texorpdfstring{$n$}{n}}
Let $\Av$ denote the average over the $n^{11}$ characters. Any eleven of the twelve coordinates $a_i$ are independent and uniform on $\{0,\dots,n-1\}$, hence so are any four of them, and any two, since any two coordinates lie in a set of eleven. In particular the $a_i$ are pairwise uncorrelated, so $\sum_ia_i$ has the same first two moments as a sum of twelve independent uniforms, namely mean $6(n-1)$ and variance $n^2-1$. (Equivalently: writing the average over the characters as $\frac1{n^{12}}\sum_{j=0}^{n-1}\sum_a\omega^{j\sum_ia_i}(-)$ with $\omega$ a primitive $n$-th root of unity, every term with $j\ne0$ carries at least ten factors $\sum_{a_i}\omega^{ja_i}=0$ when applied to $\sum_ia_i$ or to $(\sum_ia_i)^2$, so only $j=0$ contributes and the constraint may be ignored.) Thus
\[
\Av(d_a)=\frac{6(n-1)}n,\qquad \Av(d_a^2)=\frac{37n^2-72n+35}{n^2}.
\]
For a quadruple point $p$ put $s_p=\sum_{i\in T_p}a_i$ and let $r_p$ be the residue of $s_p$ modulo $n$, so that $b_p=(s_p-r_p)/n$. Adding to any one coordinate $a_i$, $i\in T_p$, the sum of the three others, which is independent of $a_i$ and uniform modulo $n$, shows that $r_p$ is uniform on $\{0,\dots,n-1\}$ and independent of each single $a_i$, $i\in T_p$. The variable $r_p$ is of course not independent of $s_p$, being a function of it; but $s_p$ is the sum of four variables each of which is independent of $r_p$, so by linearity $\Av(s_pr_p)=\sum_{i\in T_p}\Av(a_i)\Av(r_p)=4\cdot\tfrac{n-1}2\cdot\tfrac{n-1}2=(n-1)^2$. With $\Av(s_p)=2(n-1)$, $\Av(s_p^2)=\frac{n^2-1}3+4(n-1)^2$ and $\Av(r_p^2)=\frac{(n-1)(2n-1)}6$ this gives
\[
\Av(b_p)=\frac{3(n-1)}{2n},\qquad \Av(b_p^2)=\frac{16n^2-27n+11}{6n^2}.
\]
For $D=uL-\sum_pm_pE_p$ on $Z$, Riemann--Roch gives $\chi(Z,\cO_Z(D))=1+\tfrac12\bigl(u^2+3u-\sum_p(m_p^2+m_p)\bigr)$. Applying this to $D=-L_a$ and averaging, \eqref{eq:eigen} at $D_Z=0$ yields
\[
\chi(\cO_{Y_n})=n^{11}\Bigl(1+\tfrac12\bigl(\Av(d_a^2)-3\Av(d_a)-9\Av(b_p^2)+9\Av(b_p)\bigr)\Bigr)=\frac{21n^2-54n+37}{4}\,n^9,
\]
and therefore, by Noether's formula,
\begin{equation}\label{eq:alln}
\begin{gathered}
c_2(Y_n)=(18n^2-54n+48)\,n^9,\qquad \sigma(Y_n)=K_{Y_n}^2-8\chi(\cO_{Y_n})=(3n^2-11)\,n^9,\\
K_{Y_n}^2-3c_2(Y_n)=-9(n-3)^2n^9 .
\end{gathered}
\end{equation}
In particular $H_n^2=7n^9<\sigma(Y_n)$ exactly for $n\ge3$, and $Y_n$ is a ball quotient only for $n=3$. As an independent check of $c_2$, count Euler numbers: the branch divisor $B\subset Z$ has $21$ smooth rational components and $48$ nodes (the twelve double points of the arrangement and the $36$ points $\tilde L_i\cap E_p$), so $e(Z\smallsetminus B)=18$ and $e(B\smallsetminus\{\text{nodes}\})=-54$; the fibres of $f$ have $n^{11}$, $n^{10}$ and $n^9$ points over $Z\smallsetminus B$, over the smooth locus of $B$ and over the nodes, whence $e(Y_n)=18n^{11}-54n^{10}+48n^9$, in agreement with \eqref{eq:alln}. For the first exponents:
\[
\begin{array}{c|ccccc}
n & H_n^2/n^9 & K_{Y_n}\cdot H_n/n^9 & K_{Y_n}^2/n^9 & \chi(\cO_{Y_n})/n^9 & \sigma(Y_n)/n^9\\ \hline
3 & 7 & 33 & 144 & 16 & 16\\
4 & 7 & 51 & 351 & 157/4 & 37\\
5 & 7 & 69 & 648 & 73 & 64
\end{array}
\]

\subsection*{A closed formula for \texorpdfstring{$n=3$}{n=3}}
Let $n=3$. The $81$ values of the four coordinates in $T_p$ give the distribution $(15,51,15)$ for $b_p=0,1,2$, and $\sum_ia_i$ has mean $12$ and variance $8$, so the averages above specialize to
\[
\Av(d_a)=4,\qquad\Av(d_a^2)=\frac{152}{9},\qquad\Av(b_p)=1,\qquad\Av(b_p^2)=\frac{37}{27}.
\]
Applying Riemann--Roch on $Z$ to $D_s(a):=sH_Z-L_a=(4s-d_a)L-\sum_p(s-b_p)E_p$ and averaging,
\begin{equation}\label{eq:closed}
\sum_a\chi\bigl(Z,\cO_Z(D_s(a))\bigr)=3^{11}\Bigl(\frac{16}{9}+\frac{7s^2-11s}{2}\Bigr)=\lambda\Bigl(16+\frac{63s^2-99s}{2}\Bigr)\qquad(s\in\ZZ).
\end{equation}
By \eqref{eq:eigen} and $f^*H_Z=3H$, the left side is $\chi(Y_3,\cO_Y(3sH))$; the right side is $\chi(\cO_Y)+\tfrac{3s}{2}H\cdot(3sH-K)$ with $\chi(\cO_Y)=16\lambda$, $H^2=7\lambda$, $K\cdot H=33\lambda$. Thus \eqref{eq:closed} recovers $\chi(\cO_{Y_3})=16\cdot3^9$ and confirms $H^2$ and $K\cdot H$; Noether's formula then gives $c_2=48\lambda$, and $K^2=3c_2$. An explicit twist on which the failure of $\cO_{Y_3}$ to be aCM is visible: Riemann--Roch gives $\chi(\cO_{Y_3}(2H_3))=16\lambda+\tfrac12(4\cdot7-2\cdot33)\lambda=-3\lambda$, so $h^1(Y_3,\cO_{Y_3}(2H_3))\ge3\cdot3^9=59049$.

\subsection*{Irregularity and \texorpdfstring{$h^0(3A)$}{h0(3A)}}
By \eqref{eq:eigen} at $D_Z=0$, $q(Y_3)=\sum_ah^1(Z,L_a^{-1})$, a sum of superabundances of linear systems of plane curves through subsets of the nine inflection points of the Hesse pencil. For $D=-L_a$, $a\ne0$, one has $h^0(Z,D)=0$, and by Serre duality $h^2(Z,D)=h^0(Z,K_Z+L_a)$ is the dimension of the space of plane curves of degree $d_a-3$ through the points $p$ with $b_p=2$, the exceptional curves with positive coefficient in $K_Z+L_a$ being fixed components. The value $q(Y_3)=154$ is due to Ishida \cite{Is} and Naie \cite[Prop.~5.9]{Na} (Prop.~4.10 in the arXiv version).
It is accounted for by three explicit families of characters, each of whose members contributes $h^1(Z,L_a^{-1})=1$ by the Riemann--Roch formula above ($h^0=0$, and $\chi=-1$, $-1$, $+1$ with $h^2=0$, $0$, $2$ respectively):
\[
q(Y_3)=90+54+10,
\]
the three families being those with $(d_a,\#\{p:b_p=2\})=(2,1),(4,4),(6,9)$: all the weight on the four lines through one quadruple point ($9\cdot10$ characters); $a=2$ on the six lines joining four flexes no three of which are collinear, a Ceva subarrangement ($\binom94-12\cdot6=54$ such quadruples, and $b_q\le1$ at the other five flexes); $a=2$ on three of the four triangles of the arrangement, or $a=2$ on two triangles and $a=1$ on the other two ($4+6$ characters, $h^2$ being the Hesse pencil itself). Since these $154$ characters already account for the value $154$, every other character has $h^1(Z,L_a^{-1})=0$.
Hence $p_g(Y_3)=\chi-1+q=315081$. Neither $q$ nor $p_g$ enters the proofs; only $\chi$, $K^2$, $H^2$ and $K\cdot H$ do. Note, however, that the first family exists for every $n\ge2$: a character supported on $T_p$ with $\sum_ia_i=2n$ (e.g.\ $(n-1,n-1,2,0)$ for $n\ge3$, $(1,1,1,1)$ for $n=2$) has $L_a=2L-2E_p$ and $h^1(Z,L_a^{-1})=1$, so $q(Y_n)\ge1$ for every $n\ge2$. Thus $H^1(Y_n,\cO_{Y_n})\ne0$, and the failure of $R(Y_n,H)$ to be Cohen--Macaulay is elementary for every polarization $H$; the content of Corollary~\ref{cor:C} is the absence of all graded maximal Cohen--Macaulay modules.

Finally $h^0(Y_3,\cO(3A))=h^0(Y_3,f^*L)=\sum_ah^0(Z,L-L_a)$. For $d_a\ge2$ the coefficient of $L$ in $L-L_a$ is negative and $L$ is nef, so there are no sections; the trivial character contributes $h^0(Z,L)=3$; for $d_a=1$ the divisor $\sum_pb_pE_p$ has exactly one section (each $E_p$ with positive coefficient is a fixed component), and there are $\#\{a\in\{0,1,2\}^{12}:\sum_ia_i=3\}=\binom{14}{3}-12=352$ such characters. Hence $h^0(3A)=355=\binom{14}{3}-9$, in agreement with the degree-$3$ part of the coordinate ring of the complete intersection $X_3$.

\begin{remark}[more examples]\label{rem:more}
Since $\sigma(Y_n)/H_n^2=(3n^2-11)/7$ is unbounded, for fixed large $n$ one also obtains section rings $R(Y_n,H)$ with the same property for many polarizations $H$: all those with $|H|$ base-point-free, $\varphi_H$ finite and $H^2<(3n^2-11)n^9$, in particular the multiples $mH_n$ with $7m^2<3n^2-11$, which are base-point-free and ample by Proposition~\ref{prop:bpf}. The corresponding rings $R(Y_n,mH_n)=R(Y_n,H_n)^{(m)}$ are the Veronese subrings of $R(Y_n,H_n)$. Since $\sigma(Y_n)=(3n^2-11)n^9$ is strictly increasing, the surfaces $Y_n$ are pairwise non-homeomorphic, and the rings $R(Y_n,H_n)$, whose $\Proj$ recover the $Y_n$, are pairwise non-isomorphic as graded rings.

The same method applies to Hirzebruch's other classical family \cite{Hi}, treated in \cite{An} as well: the arrangement of the six lines through the pairs of four general points (four triple points, three double points), with $Z'=\mathrm{Bl}_4\PP^2$ the quintic del Pezzo surface, $-K_{Z'}=3L-\sum_pE_p$ ample, and $f'\colon Y'_n\to Z'$ the smooth $(\ZZ/n)^5$-cover branched along the six strict transforms and the four exceptional curves. With $A'=\tfrac1nf'^*L$ and $E'=\tfrac1nf'^*\sum_pE_p$, so that $A'^2=n^3$, $E'^2=-4n^3$, Hurwitz gives $K_{Y'_n}=(3n-6)A'+(2-n)E'$ and $K_{Y'_n}^2=5(n-2)^2n^3$, the count of Euler numbers (ten rational components of the branch divisor, fifteen nodes) gives $c_2(Y'_n)=(2n^2-10n+15)n^3$, hence $\chi(\cO_{Y'_n})=\tfrac1{12}(7n^2-30n+35)n^3$ and $\sigma(Y'_n)=\tfrac13(n^2-10)n^3$. The divisor $H'=3A'-E'=\tfrac1nf'^*(-K_{Z'})$ is ample, and base-point-free: $|A'|$ is pulled back from the Kummer model in $\PP^5$, and $|2A'-E'|$ is base-point-free because the product of the two coordinates of a pair of opposite sides, which together contain the four points once each, is the canonical section of $\cO(E')$ times a section with divisor $D_i+D_j$, and two such divisors for different pairs are disjoint (sides from different pairs meet at a blown-up point). Since $H'^2=5n^3<\sigma(Y'_n)$ exactly for $n\ge6$, the rings $R(Y'_n,H')$, $n\ge6$, have no graded maximal Cohen--Macaulay module either (Corollary~\ref{cor:qgor} and Remark~\ref{rem:qgor}).
\end{remark}

\begin{corollary}[a quasi-Gorenstein family]\label{cor:qgor}
For every $n\ge6$ the section ring $R=R(Y'_n,H')$ of the second Hirzebruch family of Remark~\textup{\ref{rem:more}} is a three-dimensional normal graded \emph{quasi-Gorenstein} $\CC$-domain, with $\omega_R\cong R(n-2)$ and an isolated non-Cohen--Macaulay singularity at the vertex, admitting no nonzero finitely generated graded maximal Cohen--Macaulay module. Its smallest member, $n=6$, has
\[
H'^2=1080<\sigma(Y'_6)=1872,\quad K_{Y'_6}^2=17280,\quad c_2(Y'_6)=5832,\quad \chi(\cO_{Y'_6})=1926,\quad \omega_R\cong R(4).
\]
\end{corollary}

\begin{proof}
Write $D_i'$ for the reduced ramification divisor above the strict transform of the $i$th line of the six-line arrangement. The coordinate sections used in Remark~\ref{rem:more} give $\sum_iD_i'\sim6A'-3E'$, each of the four triple points lying on three lines. The integral Hurwitz formula therefore yields
\[
K_{Y'_n}\sim n(-3A'+E')+(n-1)\Bigl(\sum_iD_i'+E'\Bigr)
\sim(n-2)(3A'-E')=(n-2)H'.
\]
This is a linear equivalence of integral divisors, so it gives $\cO_{Y'_n}(K_{Y'_n})\cong\cO_{Y'_n}((n-2)H')$. By the formula $\omega_R=\bigoplus_{t\in\ZZ}H^0\bigl(Y'_n,\cO(K_{Y'_n}+tH')\bigr)$ for the graded canonical module of a section ring \cite[(5.1.8)]{GoWa}, \cite{Wa}, which needs no Cohen--Macaulay hypothesis, $\omega_R\cong R(n-2)$. Indeed, graded local duality, \eqref{eq:lc} and Serre duality give
\[
(\omega_R)_t\cong H^3_{R_+}(R)_{-t}^{\vee}
\cong H^2(Y'_n,\cO_{Y'_n}(-tH'))^{\vee}
\cong H^0(Y'_n,\omega_{Y'_n}(tH')),
\]
where $(-)^{\vee}$ denotes $\CC$-linear duality; these isomorphisms respect multiplication by naturality. The remaining assertions are Theorem~\ref{thm:B} applied to $(Y'_n,H')$, whose hypotheses hold for $n\ge6$ by Remark~\ref{rem:more}; the displayed numbers are the formulas of that remark at $n=6$.
\end{proof}

\begin{remark}\label{rem:qgor}
Quasi-Gorenstein does not imply Gorenstein here, these rings not being Cohen--Macaulay. The rings $R(Y_n,H_n)$ of the Hesse family are not quasi-Gorenstein, so the two families differ in this respect. Indeed, writing $R=R(Y_n,H_n)$, quasi-Gorensteinness would give $\omega_R\cong R(a)$ for some $a\in\ZZ$. Even quasi-Gorensteinness at $\mathfrak m=R_+$ implies $\dim_{\CC}\omega_R/\mathfrak m\omega_R=1$; a homogeneous lift of a generator of this quotient generates $\omega_R$ by graded Nakayama, and the resulting surjection $R(a)\to\omega_R$ is an isomorphism, since both modules have rank one and $R$ is a domain. Sheafifying the canonical-module formula of \cite[(5.1.8)]{GoWa} applied to $(Y_n,H_n)$ would yield $\cO_{Y_n}(K_{Y_n})\cong\cO_{Y_n}(aH_n)$. The classes $A_n,E_n$ are numerically independent: their intersection matrix from \eqref{eq:hesse-intersections} has determinant $-9n^{18}\ne0$. Comparing coefficients in
\[
K_{Y_n}=(9n-12)A_n+(3-2n)E_n,\qquad H_n=4A_n-E_n
\]
would therefore force $a=2n-3$ and $4a=9n-12$, hence $n=0$, a contradiction for $n\ge3$.

The second family is irregular as well: with $T_p$ now the three lines through a triple point $p$, a character supported on $T_p$ with $\sum_ia_i=2n$ -- for instance $(n-1,n-1,2)$, which requires $n\ge3$ -- has $d_a=2$ and $b_p=2$, while $b_q=0$ at each of the other three triple points $q$, every one of which lies on exactly one line through $p$; so $L_a=2L-2E_p$, and Riemann--Roch on $Z'$ gives $\chi(Z',-L_a)=-1$ with $h^0(Z',-L_a)=0$ and $h^2(Z',-L_a)=h^0\bigl(Z',-L+\sum_{q\ne p}E_q-E_p\bigr)=0$, whence $h^1(Z',-L_a)=1$ and $q(Y'_n)\ge1$. As for the Hesse family, the failure of $R(Y'_n,H')$ to be Cohen--Macaulay is therefore elementary; the content of Remark~\ref{rem:more} is the absence of all graded maximal Cohen--Macaulay modules.
\end{remark}

\section{Higher dimension}\label{sec:higher}

In this section $X$ is a smooth connected projective variety of dimension $m\ge2$ over $\bk$, and $H$ is a divisor on $X$ such that $|H|$ is base-point-free with finite associated morphism, so that $H$ is ample; a vector bundle $\cE$ on $X$ is aCM with respect to $H$ if $H^i(X,\cE(tH))=0$ for $1\le i\le m-1$ and all $t\in\ZZ$. The numerical obstruction below was proved for Ulrich bundles, and for $H$ very ample, by Lopez \cite[Lem.~3.1]{Lo}. His argument runs on $X$ itself: for an Ulrich bundle $\cE$ the class $\alpha=c_1(\cE)-\tfrac{\rk\cE}2\bigl(K_X+(m+1)H\bigr)$ satisfies $\alpha\cdot H^{m-1}=0$, the number $\Delta(\cE)\cdot H^{m-2}$ differs from $\alpha^2\cdot H^{m-2}$ by a positive multiple of $\bigl(2c_2(X)-K_X^2+(m+1)H^2\bigr)\cdot H^{m-2}$, a surface section enters only to give $\alpha^2\cdot H^{m-2}\le0$ by the Hodge index theorem, and Bogomolov's inequality applies because an Ulrich bundle is semistable. The required Chern-class relations and semistability are unavailable for arbitrary aCM bundles; restricting to a smooth surface allows us instead to apply Theorem~\ref{thm:A}, which needs neither. Throughout this section, identities of cycle classes used in intersection computations are understood modulo numerical equivalence. We abbreviate $P_X:=K_X^2-2c_2(X)$; since $c_1(T_X)=-K_X$, this class is the Newton class $2\operatorname{ch}_2(T_X)$, hence additive on short exact sequences of vector bundles; it is that additivity, used in the same way in \cite[\S2]{Lo}, which makes the three computations below -- for a hyperplane section, for a product, and for a complete intersection -- immediate, and which causes the mixed terms to cancel in each of them.

\begin{proposition}\label{prop:higher}
If $(X,H)$ carries a nonzero aCM vector bundle, then
\begin{equation}\label{eq:higher}
(m+1)\,H^m\;\ge\;\bigl(K_X^2-2c_2(X)\bigr)\cdot H^{m-2}.
\end{equation}
For $m=2$ this is Theorem~\ref{thm:A}, since $K_Y^2-2c_2(Y)=3\sigma$.
\end{proposition}

\begin{proof}
Let $\cE$ be a nonzero aCM bundle and $m\ge3$. A general member $D\in|H|$ is smooth by Bertini's theorem for base-point-free linear systems in characteristic zero -- apply generic smoothness \cite[III.10.7]{Ha} to the incidence variety $\{(x,\Lambda):\varphi_H(x)\in\Lambda\}$, a projective bundle over $X$ and hence smooth, over the dual projective space; for $H$ very ample this is \cite[III.10.9]{Ha} -- and connected \cite[III.7.9]{Ha}, $H$ being ample and $X$ normal, hence a smooth connected projective variety of dimension $m-1$; $|H|_D|$ is base-point-free with finite associated morphism, and $\cE|_D$ is a nonzero aCM bundle on $(D,H|_D)$: in the cohomology sequence of $0\to\cE((t-1)H)\to\cE(tH)\to\cE|_D(tH)\to0$, the group $H^i(D,\cE|_D(tH))$, $1\le i\le m-2$, sits between $H^i(X,\cE(tH))$ and $H^{i+1}(X,\cE((t-1)H))$, both zero. Iterating, one reaches a smooth connected surface $S$, the intersection of $m-2$ general members of $|H|$, on which $\cE|_S$ is a nonzero aCM bundle with respect to $H_S:=H|_S$, and $|H_S|$ is base-point-free with finite morphism. Put $q=m-2$ and $d=H^m=H_S^2$. By adjunction and the normal bundle sequence $0\to T_S\to T_X|_S\to\cO_S(H)^{\oplus q}\to0$,
\[
K_S=(K_X+qH)|_S,\qquad c_2(S)=\Bigl(c_2(X)+q\,K_X\cdot H+\tfrac{q(q+1)}2H^2\Bigr)\cdot H^{m-2},
\]
whence $K_S^2-2c_2(S)=\bigl(K_X^2-2c_2(X)-qH^2\bigr)\cdot H^{m-2}$ -- equivalently, and without computing $c_2(S)$, the normal bundle sequence gives $2\operatorname{ch}_2(T_S)=\bigl(2\operatorname{ch}_2(T_X)-qH^2\bigr)|_S$ -- and $\sigma(S)=\tfrac13(K_S^2-2c_2(S))$. Theorem~\ref{thm:A} on $(S,H_S)$ gives $3d\ge(K_X^2-2c_2(X))\cdot H^{m-2}-(m-2)d$, which is \eqref{eq:higher}.
\end{proof}

Lemmas~\ref{lem:ring} and \ref{lem:mcm} hold for $(X,H)$ with the same proofs, $3$ being replaced by $m+1=\dim R(X,H)$: in Lemma~\ref{lem:ring} nothing changes except that $\bk[R_1]$ is now the homogeneous coordinate ring of the $m$-dimensional variety $\varphi_H(X)$, so that $\dim R=m+1$; in Lemma~\ref{lem:mcm}, \eqref{eq:lc} holds verbatim, the condition $\depth_{\mathfrak m}M=m+1$ forces $M=\Gamma_*(\widetilde M)$, $H^i(X,\widetilde M(t))=0$ for $1\le i\le m-1$ and $\Supp M=\Spec S$, for a closed point $x\in X$ the local rings $B=\cO_{X,x}$ and $S_{\mathfrak p}$ are regular of dimension $m$, and in the converse the Koszul argument uses the vanishing of $H^i\bigl(X,\bigwedge^{i+1}V\otimes\cE((t-i)H)\bigr)$ for $1\le i\le m$ and $t\gg0$. Hence:

\begin{theorem}\label{thm:higher}
If $(m+1)H^m<(K_X^2-2c_2(X))\cdot H^{m-2}$, then the section ring $R(X,H)$, a normal graded domain of dimension $m+1$ with smooth punctured spectrum, admits no graded maximal Cohen--Macaulay module and is not Cohen--Macaulay. If $H$ is very ample, its homogeneous coordinate ring also has dimension $m+1$ and smooth punctured spectrum, is not Cohen--Macaulay, and admits no nonzero finitely generated graded maximal Cohen--Macaulay module.
\end{theorem}

\begin{proof}
By Lemma~\ref{lem:mcm} in the form just described, a graded maximal Cohen--Macaulay module would give an aCM bundle of positive rank, excluded by Proposition~\ref{prop:higher}; the ring itself would be such a module if it were Cohen--Macaulay. Normality and the smoothness of the punctured spectrum are Lemma~\ref{lem:ring}(iii); for the homogeneous coordinate ring $A$ note again that $A_s=R_s$ for $s\in A_1$.
\end{proof}

\begin{corollary}\label{cor:products}
Let $m\ge2$ and $n\ge3$ with $7(m-1)^2<3n^2-11$, and let $X=Y_n\times\PP^{m-2}$ with $H=H_n\boxtimes\cO_{\PP^{m-2}}(1)$. Then $|H|$ is base-point-free with finite morphism and $(m+1)H^m<(K_X^2-2c_2(X))\cdot H^{m-2}$. Consequently the Segre product $R(X,H)=R(Y_n,H_n)\,\#\,\CC[s_0,\dots,s_{m-2}]$ is a normal generalized Cohen--Macaulay $\NN$-graded $\CC$-domain of dimension $m+1$ and depth $2$, with an isolated non-Cohen--Macaulay singularity at the vertex, admitting no nonzero finitely generated graded maximal Cohen--Macaulay module. For $m=3$ one may take any $n\ge4$, for $m=4$ any $n\ge5$.
\end{corollary}

\begin{proof}
$|H|$ is base-point-free because $|H_n|$ and $|\cO(1)|$ are, and $H\cdot C=H_n\cdot\mathrm{pr}_{1*}C+h\cdot\mathrm{pr}_{2*}C>0$ for every curve $C\subset X$, $H_n$ and $h$ being ample and $C$ not contracted by both projections; so $\varphi_H$ is finite. Write $d_n=H_n^2$, $\sigma_n=\sigma(Y_n)$, $q=m-2$, and $h$ for the hyperplane class of $\PP^q$. Then $H^m=\binom m2d_n$, and $T_X=\mathrm{pr}_1^*T_{Y_n}\oplus\mathrm{pr}_2^*T_{\PP^q}$ gives, by the additivity of $2\operatorname{ch}_2$, $K_X^2-2c_2(X)=3\sigma_n[\mathrm{pt}_{Y_n}]+(q+1)h^2$ with no mixed term, hence
\[
\bigl(K_X^2-2c_2(X)\bigr)\cdot H^{m-2}=3\sigma_n+(m-1)\binom{m-2}2d_n ,\qquad (m+1)H^m=(m+1)\binom m2d_n .
\]
Since $(m+1)\binom m2-(m-1)\binom{m-2}2=3(m-1)^2$, the inequality $(m+1)H^m<(K_X^2-2c_2(X))\cdot H^{m-2}$ is equivalent to $(m-1)^2d_n<\sigma_n$, i.e.\ to $7(m-1)^2<3n^2-11$ by \eqref{eq:alln}. Theorem~\ref{thm:higher} applies. The identification with the Segre product is the K\"unneth formula $H^0(X,\cO_X(tH))=H^0(Y_n,\cO(tH_n))\otimes H^0(\PP^{m-2},\cO(t))$; and $R$ is generalized Cohen--Macaulay because $H^i_{\mathfrak m}(R)_t\cong H^{i-1}(X,\cO_X(tH))$ vanishes for $|t|\gg0$ when $2\le i\le m$, by \eqref{eq:lc} in dimension $m$ and Serre vanishing and duality, while $H^0_{\mathfrak m}(R)=H^1_{\mathfrak m}(R)=0$ by normality. Finally $H^2_{\mathfrak m}(R)_0\cong H^1(X,\cO_X)=H^1(Y_n,\cO_{Y_n})$ by the K\"unneth formula, which is nonzero since $q(Y_n)\ge1$ (Section~\ref{sec:example}); with normality this gives $\depth_{\mathfrak m}R=2$.
\end{proof}

\begin{corollary}[the characteristic hypothesis in \cite{ShTa} is essential]\label{cor:shta}
Let $n\ge4$ and let $X=Y_n\times\PP^1$ with $H=H_n\boxtimes\cO_{\PP^1}(1)$, as in Corollary~\ref{cor:products} with $m=3$. Then $X$ is a smooth projective threefold over $\CC$, hence normal of dimension $3$ and locally factorial; $H$ is an ample base-point-free integral Cartier divisor, so that $\cO_X(iH)$ is invertible for every $i\in\ZZ$; the structure sheaf $\cO_X$ is a coherent sheaf all of whose stalks are maximal Cohen--Macaulay; and $R(X,H)$ admits no nonzero finitely generated graded maximal Cohen--Macaulay module. Thus every hypothesis of \textup{\cite[Cor.~4.5(i)]{ShTa}} is satisfied once its positive-characteristic, $F$-finite base field is replaced by an algebraically closed field of characteristic zero, and the implication from the existence of a maximal Cohen--Macaulay sheaf on $X$ to the existence of a graded maximal Cohen--Macaulay module over $R(X,H)$ fails in characteristic zero.
\end{corollary}

\begin{proof}
All the assertions about $(X,H)$ are contained in Corollary~\ref{cor:products}, which applies because $7(m-1)^2=28<3n^2-11$ for $n\ge4$; a smooth variety is locally factorial, and the stalks $\cO_{X,x}$ are regular local rings, hence maximal Cohen--Macaulay modules over themselves.
\end{proof}

\begin{remark}[other factors]\label{rem:factors}
Nothing in the proof requires the second factor to be a projective space; products of a surface with an arbitrary variety are the form in which \cite[Thm.~1]{Lo} obtains varieties of every dimension carrying no Ulrich bundle, for a very ample polarization on each factor. Let $(Z,D)$ be a smooth connected projective variety of dimension $q=m-2$ with $|D|$ base-point-free and $\varphi_D$ finite, and put $X=Y_n\times Z$, $H=H_n\boxtimes D$ and $\tau_Z=(K_Z^2-2c_2(Z))\cdot D^{q-2}$, with the convention $\tau_Z=0$ for $q\le1$, where the coefficient $\binom{m-2}2$ below vanishes. Since $K_X^2-2c_2(X)=3\sigma_n[\mathrm{pt}_{Y_n}]+(K_Z^2-2c_2(Z))$, one gets $H^m=\binom m2d_nD^q$ and $(K_X^2-2c_2(X))\cdot H^{m-2}=3\sigma_nD^q+\binom{m-2}2d_n\tau_Z$, so that the inequality of Proposition~\ref{prop:higher} fails exactly when
\[
3\,\frac{\sigma_n}{d_n}\ >\ (m+1)\binom m2-\binom{m-2}2\frac{\tau_Z}{D^q} .
\]
For $Z=\PP^{m-2}$ this is $(m-1)^2d_n<\sigma_n$ again. The right-hand side is a constant once $(Z,D)$ is fixed, while $\sigma_n/d_n=(3n^2-11)/7$ is unbounded, so every such $(Z,D)$ occurs for $n\gg0$; the base-point-freeness, the finiteness of $\varphi_H$, the K\"unneth identification and the computation of the depth are unchanged, the last because $H^1(X,\cO_X)=H^1(Y_n,\cO_{Y_n})\oplus H^1(Z,\cO_Z)\ne0$. The choice $Z=\PP^{m-2}$ is the one for which $R(Z,D)$ is a polynomial ring, so that $R(X,H)$ is a Segre product with a polynomial ring, as in \cite{Ma19}.
\end{remark}

\begin{corollary}[complete intersections]\label{cor:ci}
Let $X$ be a smooth connected projective variety of dimension $m$ over $\bk$ and $H$ a divisor with $|H|$ base-point-free and $\varphi_H$ finite. Let $1\le r\le m-2$, let $a_1,\dots,a_r\ge1$, let $X_r\subset X$ be the intersection of general members of $|a_1H|,\dots,|a_rH|$, and put $H_r=H|_{X_r}$. If
\begin{equation}\label{eq:ci}
\Bigl(m+1-r+\sum_{i=1}^ra_i^2\Bigr)H^m<\bigl(K_X^2-2c_2(X)\bigr)\cdot H^{m-2},
\end{equation}
then $(X_r,H_r)$ carries no nonzero aCM bundle, and $R(X_r,H_r)$ is a normal graded domain of dimension $m-r+1$ with smooth punctured spectrum which is not Cohen--Macaulay and admits no nonzero finitely generated graded maximal Cohen--Macaulay module.
\end{corollary}

\begin{proof}
Each $|a_iH|$ is base-point-free with finite associated morphism, $a_iH$ being ample and base-point-free; so the members cut out successively are smooth by the Bertini argument in the proof of Proposition~\ref{prop:higher} and connected by \cite[III.7.9]{Ha}, the dimension being at least $2$ at every step, and $|H_r|$ is base-point-free with finite associated morphism by restriction. For a single cut $X_1\in|a_1H|$, adjunction and the normal bundle sequence give $K_{X_1}=(K_X+a_1H)|_{X_1}$ and $c_2(X_1)=\bigl(c_2(X)+a_1K_XH+a_1^2H^2\bigr)|_{X_1}$, whence $H_1^{m-1}=a_1H^m$ and
\[
\bigl(K_{X_1}^2-2c_2(X_1)\bigr)\cdot H_1^{m-3}=a_1\Bigl[\bigl(K_X^2-2c_2(X)\bigr)\cdot H^{m-2}-a_1^2H^m\Bigr];
\]
so \eqref{eq:higher} fails for $(X_1,H_1)$ exactly when $(m+a_1^2)H^m<(K_X^2-2c_2(X))\cdot H^{m-2}$. For Ulrich bundles the case $r=1$ is the computation of \cite[\S4]{Lo}, and the case $a_1=\dots=a_r=1$ is the equivalence proved in \cite[Lem.~3.1]{Lo}, by which the obstruction descends to all successive hyperplane sections; what the iteration adds is the general multidegree. Iterating, $H_r^{m-r}=\bigl(\prod_ia_i\bigr)H^m$ and $(K_{X_r}^2-2c_2(X_r))\cdot H_r^{m-r-2}=\bigl(\prod_ia_i\bigr)\bigl[(K_X^2-2c_2(X))\cdot H^{m-2}-\bigl(\sum_ia_i^2\bigr)H^m\bigr]$, so that \eqref{eq:higher} fails for $(X_r,H_r)$ exactly under \eqref{eq:ci}. No terms $a_ia_j$ with $i\ne j$ occur: they cancel because $K^2-2c_2=2\operatorname{ch}_2(T)$ is additive. Note also that the restriction of an aCM bundle to $X_r$ is again aCM, so existence on $X$ implies existence on $X_r$; it is the opposite implication that is wanted here, and it does not follow, which is why \eqref{eq:ci} has to be verified on $X_r$ itself. Theorem~\ref{thm:higher} applies.
\end{proof}

\begin{remark}[new surfaces]\label{rem:ci}
For the ambient $X=\PP^N$ with $H=\cO(1)$, condition \eqref{eq:ci} reads $\sum_ia_i^2<r$, which never holds since each $a_i\ge1$. For $H=\cO(b)$ it reads $\bigl(N+1-r+\sum_ia_i^2\bigr)b^2<N+1$, again impossible. For $X=Y_n\times\PP^{m-2}$ and $H=H_n\boxtimes\cO(1)$ it becomes
\[
\sigma_n>\Bigl[(m-1)^2+\tfrac13\Bigl(\sum_ia_i^2-r\Bigr)\tbinom m2\Bigr]d_n ,
\]
which holds for $n\gg0$ since $\sigma_n/d_n=(3n^2-11)/7$ is unbounded.

The iteration is transparent in terms of the excess $\mathcal I(X,H):=\bigl(K_X^2-2c_2(X)\bigr)\cdot H^{m-2}-(m+1)H^m$, by which \eqref{eq:higher} reads $\mathcal I(X,H)\le0$ and \eqref{eq:ci} reads $\mathcal I(X_r,H_r)>0$: the computation in the proof of Corollary~\ref{cor:ci} says
\[
\mathcal I(X_r,H_r)=\Bigl(\prod_ia_i\Bigr)\Bigl[\mathcal I(X,H)-\sum_i(a_i^2-1)H^m\Bigr] .
\]
Hyperplane cuts, all $a_i=1$, therefore preserve the excess exactly -- for Ulrich bundles this invariance is the equivalence in \cite[Lem.~3.1]{Lo} -- while each $a_i\ge2$ costs a strictly positive amount: the obstruction does not propagate of itself to sections of higher degree.

Taking $r=m-2$ one obtains smooth \emph{surfaces} $X_r$ with $H_r^2<\sigma(X_r)$, hence further three-dimensional section rings without graded maximal Cohen--Macaulay modules. None of these surfaces is a product of two curves: for $C_1\times C_2$ one has $K^2=8(g_1-1)(g_2-1)$ and $\chi(\cO)=(g_1-1)(g_2-1)$, so $\sigma=0$, whereas $\sigma(X_r)>H_r^2>0$. For a single cut in a product the same additivity settles it in any dimension. Suppose $\operatorname{ch}_{m-1}(T_X)=0$, as holds for $X=Y_n\times\PP^{m-2}$ with $m\ge4$, both factors then having dimension $<m-1$. For a smooth $D\in|aH|$ the normal bundle sequence gives $\operatorname{ch}_{m-1}(T_D)=-\bigl((aH)^{m-1}/(m-1)!\bigr)|_D$, whence
\[
(m-1)!\int_D\operatorname{ch}_{m-1}(T_D)=-a^mH^m\ne0 ;
\]
a nontrivial product of dimension $m-1$ has this number zero, both factors having dimension $<m-1$. Compare \cite[Lem.~2.1]{Lo}, where the same additivity of the Newton classes is used. That they are not isomorphic to the $Y_k$ themselves needs a separate argument, which we give in the first case only. For $m=3$ and $a_1=1$ the surface $X_1$ is the blow-up of $Y_n$ at the $H_n^2$ base points of a general pencil in $|H_n|$, with $H_1=2\pi^*H_n-E$, $H_1^2=3H_n^2$ and $\sigma(X_1)=\sigma_n-H_n^2$, so that \eqref{eq:ci} reduces to the condition $7(m-1)^2<3n^2-11$ of Corollary~\ref{cor:products} for $m=3$; in particular $X_1$ is not minimal, whereas every $K_{Y_k}$ is ample, so $X_1$ is none of the $Y_k$. For Ulrich bundles these surface sections already occur in \cite[Thm.~1]{Lo}; for $a_1\ge2$ the projection $X_1\to Y_n$ is finite of degree $a_1$: the equation of $X_1\subset Y_n\times\PP^1$ is a form of degree $a_1$ in the two coordinates of $\PP^1$ whose $a_1+1$ coefficients are general members of $H^0(Y_n,a_1H_n)$, and since $|a_1H_n|$ is base-point-free and $a_1+1>2=\dim Y_n$, general such coefficients have no common zero on $Y_n$; hence no fibre $\{y\}\times\PP^1$ is contained in $X_1$.
\end{remark}

A cheaper way to raise the dimension is the polynomial ring $R[x]$, with $\deg x=1$, over a ring $R$ without graded maximal Cohen--Macaulay modules: $x$ is a regular element on any graded maximal Cohen--Macaulay $R[x]$-module $M$, being part of a homogeneous system of parameters, and $M/xM$ is then a graded maximal Cohen--Macaulay $R$-module. The singular locus of $R[x]$ is not isolated, however, which is why the products above are used.

\section{Remarks}\label{sec:remarks}

\begin{remark}[Ulrich bundles and Ulrich modules]\label{rem:ulrich}
Ulrich bundles on $(Y,H)$ are the bundles $\cE$ with $H^i(Y,\cE(-jH))=0$ for all $i$ and $j=1,2$; equivalently, by the Hilbert-polynomial characterization \cite[Prop.~2.1 and Cor.~2.2]{ES}, the aCM bundles with $\chi(\cE(tH))=\tfrac12\rk(\cE)H^2(t+1)(t+2)$; for $H$ very ample they correspond to the graded maximal Cohen--Macaulay modules over $A(Y,H)$ whose minimal graded free resolution over the ambient polynomial ring $P=\Sym_{\bk}H^0(Y,\cO_Y(H))$ is linear and begins in degree zero. The linearity in this characterization refers to the $P$-free resolution; the $A(Y,H)$-free resolution need not be linear: when $A=A(Y,H)$ is a hypersurface ring $P/(f)$ with $\deg f=q\ge2$, for instance, the minimal $A$-free resolution of an Ulrich module is the infinite resolution, $2$-periodic up to a grading shift by $q$, associated with a matrix factorization $\varphi\psi=\psi\varphi=f\cdot\mathrm{id}$ in which $\varphi$ is linear and $\psi$ has degree $q-1$, so that its maps alternate between degrees $1$ and $q-1$ ($1$ and $2$ for a cubic surface, $1$ and $1$ for a quadric). In the present setting that equivalence is immediate from Lemma~\ref{lem:horrocks}: writing $\pi_*\cE=\bigoplus_{j=1}^N\cO(a_j)$ with $N=\rk(\cE)d$, the prescribed Hilbert polynomial forces $\sum_ja_j=0$ and $\sum_ja_j^2=0$, so $\pi_*\cE\cong\cO_{\PP^2}^{\oplus N}$ and the two vanishings follow from those of $H^i(\PP^2,\cO(-1))$ and $H^i(\PP^2,\cO(-2))$; this is the projection characterization of \cite[\S2]{ES}.

The existence of an Ulrich sheaf with full support on every projective variety was asked in \cite{ES}; see \cite{Be} for an introduction. On a surface an Ulrich bundle is semistable \cite[Thm.~2.9]{CHGS}, so that Bogomolov's inequality forces $H^2\ge K_Y^2-8\chi(\cO_Y)$ (\cite{Be17} for Picard rank one, \cite{An} in general); the non-existence of Ulrich bundles on the Hesse pairs is \cite[Thm.~1.2]{An}, proved by Bogomolov's inequality for the (semistable) Ulrich bundle itself, and Lopez \cite{Lo} obtains varieties of every dimension without Ulrich bundles from the numerical obstruction that Section~\ref{sec:higher} extends to aCM bundles. Theorem~\ref{thm:A} and Proposition~\ref{prop:higher} remove the semistability, at the cost of the majorization argument, and exclude all finitely generated graded MCM modules. An Ulrich bundle is aCM with a prescribed Hilbert polynomial and is automatically semistable; the aCM condition prescribes neither the Hilbert polynomial nor a stability property.

Accordingly, whereas the existence of Ulrich bundles is a delicate question, aCM bundles are abundant wherever they have been studied --- on cubic surfaces they have been studied in every rank and classified in rank two \cite{CH,Fa}, and on every arithmetically Cohen--Macaulay embedded variety $\cO_X$ is one, its coordinate ring being a graded maximal Cohen--Macaulay module over itself. No polarized variety outside Theorem~\ref{thm:A} and Proposition~\ref{prop:higher} is known to the author to carry no nonzero aCM bundle at all. 
Both obstructions bear on the fixed polarization and not on the surface: by Coskun--Huizenga \cite[Thm.~1.2]{CoHu}, every smooth complex projective surface carries rank-two Ulrich bundles with respect to $mH$ for all $m\gg0$, in accordance with Remark~\ref{rem:more}, where the multiples of $H_n$ excluded by Theorem~\ref{thm:A} are those with $7m^2<3n^2-11$. On a fixed Hesse surface the two statements combine into a threshold phenomenon for the Veronese subrings of a single ring. Every multiple $mH_n$ is ample and base-point-free (Proposition~\ref{prop:bpf}) and $R(Y_n,mH_n)=R(Y_n,H_n)^{(m)}$ (Remark~\ref{rem:more}). By Theorem~\ref{thm:B}, if
\[
m\ <\ \sqrt{\tfrac{3n^2-11}{7}}
\]
then $R(Y_n,H_n)^{(m)}$ has no graded maximal Cohen--Macaulay module whatsoever; whereas by \cite[Thm.~4.3]{CoHu} there is an $m_0=m_0(Y_n,H_n)$ such that for $m\ge m_0$ it has one of every even rank, and, after increasing $m_0$ if necessary, \cite[Cor.~4.5]{CoHu} shows that the pair $(Y_n,mH_n)$ is of Ulrich wild representation type, so that arbitrarily large families of pairwise non-isomorphic indecomposable ones occur. What happens between the two thresholds we do not know.

We note also that the numerical relations satisfied by an Ulrich bundle on a polarized surface, recorded in \cite[Prop.~4.2]{CoHu}, are \eqref{eq:rr} read at its vertex: the first says exactly that the vertex of $\chi(\cE(t))$ lies at $t=-\tfrac32$, which in the normalization $\nu=c_1(\cE)/\rk\cE$ of Section~\ref{sec:proofA} reads $2\nu\cdot H=3H^2+K_Y\cdot H$ (the parameter denoted by $\nu$ in \cite[Prop.~4.2]{CoHu} is the total slope of $\cE(-H)$; writing this parameter as $\nu_{\mathrm{CH}}$, one has $\nu_{\mathrm{CH}}=\nu-H$, so that the same relation is written there as $2\nu_{\mathrm{CH}}\cdot H=H^2+K_Y\cdot H$; their normalized discriminant is likewise half of ours, $\Delta_{\mathrm{CH}}=\Delta(\cE)/2\rk(\cE)^2=\delta/2$, which accounts for the apparent factor $2$ between the two sets of relations); and the second is the value of \eqref{eq:rr} there, $\delta(\cE)-\alpha(\cE)^2=\tfrac{d-\sigma}4$, the case of equality in Remark~\ref{rem:rr}. The obstruction is compatible with the existence results for aCM bundles on surfaces of Kodaira dimension $\le1$, such as \cite{CH,Fa}: there $\sigma\le1\le d$, and $\sigma\le0$ except on $\PP^2$ (Remark~\ref{rem:signature}), so Theorem~\ref{thm:A} never obstructs.
\end{remark}

\begin{remark}[cohomology tables]\label{rem:bs}
The existence question for Ulrich sheaves raised in \cite{ES} is related to Boij--S\"oderberg theory by \cite[Thm.~5.3]{ES11}: the cone of finite nonnegative rational combinations of the cohomology tables $\bigl(h^i(X,\cF(t))\bigr)_{i,t}$ of coherent sheaves on $X\subset\PP^N$ coincides with the corresponding cone for $(\PP^{\dim X},\cO(1))$ if and only if $X$ carries an Ulrich sheaf. On projective space, the tables of vector bundles admit finite positive rational decompositions into supernatural bundle tables \cite{ES09}; for arbitrary coherent sheaves, the decompositions of \cite{ES10} may be infinite, entrywise convergent series of supernatural sheaf tables, including tables of bundles on linear subspaces extended by zero. The latter decomposition theorem concerns series and must be distinguished from generation by finite combinations.
For vector-bundle tables, Theorem~\ref{thm:A} gives the following exclusion. Let $C_{vb}(Y,H)$ be the set of finite linear combinations, with nonnegative rational coefficients, of the cohomology tables $\gamma(\cE)=\bigl(h^i(Y,\cE(tH))\bigr)_{i,t}$ of vector bundles on $Y$. If $d<\sigma$, then
\[
C_{vb}(Y,H)\cap\bigl\{\gamma:\gamma_{1,t}=0\text{ for all }t\in\ZZ\bigr\}=\{0\},
\]
because all entries of a cohomology table are nonnegative, so that a combination with positive coefficients has vanishing $h^1$-row only if each of its terms does, i.e.\ only if every bundle involved is aCM. On $\PP^2$, by contrast, the tables of the line bundles $\cO(a)$ have vanishing $h^1$-row; so $C_{vb}(Y_n,H_n)$ is not the cone of $\PP^2$, and no nonzero table in it lies on the subspace of tables without intermediate cohomology. This concerns finite combinations of tables of vector bundles only: nothing is claimed about the closure of the cone, about limits of normalized tables, or about coherent sheaves, which is the setting of the cone-comparison theorem of \cite[Thm.~5.3]{ES11} just quoted.
\end{remark}

\section*{Acknowledgements}

This note owes its existence to a short email exchange with David Eisenbud and Frank-Olaf Schreyer concerning \cite{An}. The author is particularly indebted to them for the explicit questions: ``Do there exist aCM vector bundles on $Y$, i.e., $\cE$ with $h^1(\cE(d))=0$ for all $d\in\ZZ$? What is the Boij--S\"oderberg cone of cohomology tables of aCM bundles on $Y$?'' (the twist is written $t$, and $d=H^2$, in this note). Theorem~\ref{thm:A} answers the first in the negative, and Remark~\ref{rem:bs} draws the consequence for the second.

\section*{Statements and declarations}

\noindent\emph{Funding.} The author received no research funding for this work.

\noindent\emph{Competing interests.} The author has no financial interests to disclose.

\noindent\emph{Use of generative AI.} ChatGPT and Claude were used for exploratory computations, literature search and editorial assistance. The arguments of this note are given in full, with the external results used cited explicitly, and the author is solely responsible for the final text.

\end{document}